\documentclass{amsart}

\usepackage{amsmath,amsthm,amstext,amssymb,bbm}
\usepackage{hyperref}
\usepackage{accents}
\usepackage[utf8]{inputenc}
\usepackage{enumerate,enumitem}
\usepackage{amscd,  amsfonts,  amsbsy,  mathrsfs, upgreek, mathtools, stmaryrd, bbm}

\usepackage{todonotes}

\theoremstyle{plain}
\newtheorem{theorem}{Theorem}[section]
\newtheorem{lemma}[theorem]{Lemma}

\newtheorem*{observation*}{Observation}
\newtheorem{corollary}[theorem]{Corollary}
\newtheorem{proposition}[theorem]{Proposition}

\newtheorem*{claim*}{Claim}
\newtheorem*{subclaim*}{Subclaim}
\theoremstyle{definition}
\newtheorem{definition}[theorem]{Definition}
\newtheorem{remark}[theorem]{Remark}

\newtheorem{fact}[theorem]{Fact}
\newtheorem{question}[theorem]{Question}
\newtheorem*{question*}{Question}

\newcommand{\betrag}[1]{\vert{#1}\vert}

\newcommand{\dom}[1]{{{\rm{dom}}(#1)}}

\newcommand{\cof}[1]{{{\rm{cof}}(#1)}}
\newcommand{\otp}[1]{{{\rm{otp}}\left(#1\right)}}
\newcommand{\ran}[1]{{{\rm{ran}}(#1)}}

\newcommand{\tc}[1]{{\rm{tc}}({#1})}

\newcommand{\POT}[1]{{\mathcal{P}}({#1})}

\newcommand{\map}[3]{{#1}:{#2}\longrightarrow{#3}}
\newcommand{\Map}[5]{{#1}:{#2}\longrightarrow{#3};~{#4}\longmapsto{#5}}

\newcommand{\Set}[2]{\{{#1}~\vert~{#2}\}}
\newcommand{\seq}[2]{\langle{#1}~\vert~{#2}\rangle}

\newcommand{\goedel}[2]{{\prec}{#1},{#2}{\succ}}

\newcommand{\anf}[1]{{\text{``}\hspace{0.3ex}{#1}\hspace{0.3ex}\text{''}}}

\newcommand{\HH}[1]{\rm{H}_{#1}}

\newcommand{\Add}[2]{{\rm{Add}}({#1},{#2})}

\newcommand{\id}{{\rm{id}}}
\newcommand{\Lim}{{\rm{Lim}}}
\newcommand{\On}{{\rm{Ord}}}
\newcommand{\LL}{{\rm{L}}}
\newcommand{\HOD}{{\rm{HOD}}}

\newcommand{\ZFC}{{\rm{ZFC}}}

\newcommand{\PFA}{{\rm{PFA}}}

\newcommand{\uhr}{\upharpoonright}
\newcommand{\co}{\mathbb{C}}
\newcommand{\po}{\mathbb{P}}
\DeclareMathOperator{\col}{Col}

\newcommand{\QQQ}{{\mathbb{Q}}}

\newcommand{\KK}{{\rm{K}}}
\newcommand{\MM}{{\rm{M}}}

\newcommand{\VV}{{\rm{V}}}

\newcommand{\calC}{\mathcal{C}}

\newcommand{\calL}{\mathcal{L}}

 \newcommand{\GCH}{{\rm{GCH}}}

\newcommand{\power}{\mathcal{P}}

\newcommand{\qo}{\mathbb{Q}}

\newcommand{\name}{\dot}
\newcommand{\can}{\check}
\newcommand{\fr}{{}^\frown}

\makeatletter
\@namedef{subjclassname@2020}{%
  \textup{2020} Mathematics Subject Classification}
\makeatother

\title{On $\Sigma_1$-Definable Closed Unbounded Sets}

\author{Omer Ben-Neria}
\address{Einstein Institute of Mathematics, The Hebrew University of Jerusalem. The Edmond J. Safra Campus (Giv’at Ram), Jerusalem 91904, Israel}
\email{omer.bn@mail.huji.ac.il}

\author{Philipp L\"{u}cke}
\address{Fachbereich Mathematik, Universit\"at Hamburg, Bundesstra{\ss}e 55, Hamburg, 20146, Germany}
\email{philipp.luecke@uni-hamburg.de}


\thanks{The first author would like to thank the Israel Science Foundation (Grants 1832/19 and 1302/23) for their support. 
The second author gratefully acknowledges support by the \emph{Deutsche Forschungsgemeinschaft} (Project number 522490605). 
In addition, the research of the second author was supported by the Spanish Government under grant EUR2022-134032. 
This project has received funding from the European Union’s Horizon 2020 research and innovation programme under the Marie Sk{\l}odowska-Curie grant agreement No 842082 of the second author (Project \emph{SAIFIA: Strong Axioms of Infinity -- Frameworks, Interactions and Applications}).}

\begin{document}

\begin{abstract}
 Definable stationary sets, and specifically, ordinal definable ones, play a significant role in the study of canonical inner models of set theory and the class HOD of hereditarily ordinal definable sets. 
 Fixing a certain notion of definability and an uncountable cardinal, one can consider the associated family of definable closed unbounded sets.
 In this paper, we study the extent to which such families can approximate the full closed unbounded filter, and their dependence on the defining complexity.  
 Focusing on closed unbounded subsets of a cardinal $\kappa$ which are $\Sigma_1$-definable in parameters from $\HH{\kappa}$ and ordinal parameters, we show that the ability of such closed unbounded sets to well approximate the closed unbounded filter on $\kappa$ can highly vary, and strongly depends on key properties of the underlying universe of set theory. 
\end{abstract}

\maketitle



\section{Introduction}\label{section:Intro}

The concepts of closed unbounded and stationary subsets capture many of the key aspects of the combinatorics of cardinals of uncountable cofinalities. 
 Recent developments in the study of canonical inner models of set theory provide strong motivations for analyzing the definability,  specifically, the ordinal definability, of these objects. 
 In particular, the notion of a \emph{$\omega$-strongly measurable cardinal $\kappa$ in $\HOD$}, introduced by Woodin (see {\cite[Definition 189]{Woodin-SuitableExtendersI}}) to measure  local failures of the inner model $\HOD$ to approximate the set-theoretic  universe $V$, 
 is equivalent to the fact that the restriction of the closed unbounded filter on $\Set{\alpha<\kappa}{\cof{\alpha}=\omega}$ to $\HOD$  identifies with the intersection of a small number of ordinal definable normal measures on $\kappa$ in $\HOD$ (see {\cite[Lemma 2.4]{BH21}}). 
 Since, in $\HOD$, many of the sets in this intersection do not contain a closed unbounded subset,  it follows that there are many subsets of the cardinal $\kappa$ that are  $\HOD$-stationary (i.e., that meet every closed unbounded subset of $\kappa$ that is an element of $\HOD$), but are not stationary subsets of $\kappa$ in $V$.

 Motivated by this observation, the purpose of the work presented in this paper is to study the extent to which hierarchies of definable closed unbounded subsets of an uncountable cardinal $\kappa$ can approximate the full closed unbounded filter on $\kappa$, where the different hierarchies are given in terms of the  complexity of the defining formulas (e.g., $\Sigma_n$-formulas\footnote{See {\cite[p. 183]{MR1940513}}.}  for some natural number $n$) and the set of parameters allowed in these formulas (e.g., ordinal parameters or parameters contained in $\HH{\kappa}$). 
 Since ordinal definable sets of ordinals can be represented as the unique solutions of $\Sigma_2$-formulas with ordinal parameters (see the proof of {\cite[Lemma 13.25]{MR1940513}} for details) and closed unbounded subsets of uncountable cardinals $\kappa$ that are definable by $\Sigma_0$-formulas with ordinals and elements of $\HH{\kappa}$ can be shown to have a very simple, eventually periodic structure,\footnote{More precisely, it is possible to use  {\cite[Lemma 2.3]{Sigma1Partitions}} to show that, if $\varphi(v_0,\ldots,v_n)$ is a $\Sigma_0$-formula, $\kappa$ is an uncountable cardinals and $z_0,\ldots,z_{n-1}\in\HH{\kappa}\cup\On$ with the property that $C=\Set{\alpha<\kappa}{\varphi(\alpha,z_0,\ldots,z_{n-1})}$ is a   closed unbounded subset  of $\kappa$, then there exists a set $N$ of natural numbers such that $C\cap[\lambda,\lambda+\omega)=\Set{\lambda+n}{n\in N}$ holds for coboundedly many limit ordinals $\lambda<\kappa$.}  
 the first place to expect different phenomena is with $\Sigma_1$-definable closed unbounded sets.  
 In order to motivate the definitions and results contained in this paper, we start by presenting two settings in which $\Sigma_1$-definable closed unbounded sets  give an optimal approximation to the collection of all closed unbounded sets: 
 \begin{itemize}
  \item In G\"odel's constructible universe $L$, every  subset (and, in particular, every closed unbounded subset) of an infinite cardinal $\kappa$ is the unique solution of a $\Sigma_1$-formula with parameters from $\kappa^+$. 
  This follows directly from the fact that every such subset is the $\eta$-th element in the canonical well-ordering of $L$ for some $\eta<\kappa^+$, and transitive models of $\anf{\ZFC^-+{V=L}}$ containing an ordinal $\eta$ can uniformly compute the initial segment of order-type $\eta+1$ of this well-ordering.

  \item \emph{Martin's Maximum} $\MM\MM$ implies that every closed unbounded subset of  $\omega_1$ contains a closed unbounded subset  that is the unique solution of a $\Sigma_1$-formula that only uses the ordinal $\omega_1$  and real numbers as parameters. 
   This statement follows directly from results of Woodin that show that $\MM\MM$ implies \emph{admissible club guessing} 
  (see {\cite[Theorems 3.16, 3.17 \& 3.19]{MR1713438}}).  
\end{itemize}
In addition, results that will be contained in a sequel to this paper show that for every uncountable  cardinal $\kappa$ satisfying $\kappa^{{<}\kappa}=\kappa$, there is cofinality-preserving forcing extension $V[G]$ of the ground model $V$ in which there exists  a subset $E$ of $\kappa$ such that every closed unbounded subset of $\kappa$ contains a closed unbounded subset that is the unique solution of a $\Sigma_1$-formula that only uses the set $E$ and ordinals less than $\kappa^+$ as parameters.

 In  contrast to these settings, the results of this paper will isolate several canonical contexts in which  collections of $\Sigma_1$-definable closed unbounded sets fail to approximate the closed unbounded filter. In order to formulate these results, we now specify the notions of definability used in this paper.

 \begin{definition}\label{definition:definable}
  A class $\Gamma$ is \emph{definable} by a formula $\varphi(v_0,\ldots,v_n)$ and parameters $y_0,\ldots,y_{n-1}$ if $$\Gamma ~ = ~ \Set{x}{\varphi(x,y_0,\ldots,y_{n-1})}.$$
 \end{definition}

Throughout this paper, we study the definability of certain sets of ordinals. Given such a set $E$, we will mostly consider the questions that ask whether the set $\{E\}$ is definable (in the sense of Definition \ref{definition:definable}) by certain formulas with parameters coming from a given class. Note that this form of definability differs from the statement that the set $E$ itself is definable (again in the sense of Definition \ref{definition:definable}) in the given way. While the former implies the latter, it is easy to see that the converse implication fails for the set $\omega_1$ of all countable ordinals and definability by  $\Sigma_1$-formulas without parameters.

 We now introduce the main concepts studied in this paper:

\begin{definition}\label{Def:Sigma_n-Stationary}
 Let $n$ be a natural number, let $\kappa$ be a uncountable cardinal and let $S$ be a subset of $\kappa$.  
 \begin{enumerate}
   \item  Given a class $R$, the subset $S$  is \emph{$\Sigma_n(R)$-stationary in $\kappa$} if $C\cap S\neq\emptyset$ holds for every closed unbounded subset $C$ of $\kappa$ with the property that the set $\{C\}$ is definable by a $\Sigma_n$-formula with parameters in $R\cup\{\kappa\}$.

   \item The subset $S$ is \emph{$\Sigma_n$-stationary in $\kappa$} if it is $\Sigma_n(\emptyset)$-stationary in $\kappa$. 

   \item Given a class $R$, the subset $S$ is \emph{$\mathbf{\Sigma}_n(R)$-stationary in $\kappa$} if it is $\Sigma_n(R\cup\HH{\kappa})$-stationary in $\kappa$. 

   \item The subset $S$ is \emph{$\mathbf{\Sigma}_n$-stationary in $\kappa$} if it is $\mathbf{\Sigma}_n(\emptyset)$-stationary in $\kappa$.  
 \end{enumerate} 
\end{definition}

 Trivially, all stationary subsets of a given uncountable cardinal $\kappa$ are $\Sigma_n(R)$-stationary in $\kappa$ for every parameter class $R$ and every natural number $n$. 
 As indicated by Definition \ref{Def:Sigma_n-Stationary}, our focus in this paper will be on natural classes $R$ such as $\On$, $\HH{\kappa}$, $\HH{\mu}$ for some $\mu < \kappa$, and their combinations. 
 We therefore view the existence of non-stationary $\Sigma_n(R)$-stationary sets as a measure  for the discrepancy between the collection of  $\Sigma_n(R)$-definable closed unbounded subsets and the collection of all closed unbounded subsets. 
 Our results will isolate several settings in which highly non-stationary sets (such as singletons, or sets of successor
cardinals below a limit cardinal) are  stationary for rich collections of definable sets. 
 In particular,  even in the case  of singular cardinals of countable cofinality, where  stationarity coincides with coboundedness, we will present canonical examples of such cardinals in which  sparse  subsets are $\Sigma_n(R)$-stationary. 
 The proofs of all these results reveal that, in the given settings, that the considered collections of $\Sigma_n(R)$-definable closed unbounded subsets and $\Sigma_n(R)$-stationary sets possess many of structural features provable for the collections of closed unbounded and stationary subsets of regular uncountable cardinals. 
 Below, we list six of our main results that best illustrate the phenomenon described above: 
 
\begin{itemize}
    \item If $\kappa$ is a Ramsey cardinal, then every unbounded  subset of $\kappa$ that consists of cardinals  is $\mathbf{\Sigma}_1$-stationary (Theorem \ref{theorem:StablyMeasurable}). 

    \item If $\kappa$ is a (possibly singular) limit of measurable cardinals,  then every unbounded  subset of $\kappa$ that consists of cardinals is $\mathbf{\Sigma}_1(Ord)$-stationary (Theorem \ref{theorem:LimitMeasurables}). 

    \item If  $\omega_\omega$ is a J\'{o}nsson cardinal, then every unbounded subset of $\Set{\omega_n}{n<\omega}$ is $\mathbf{\Sigma}_1$-stationary in $\omega_\omega$ (Theorem \ref{theorem:JonssonStationary}). 

    \item It is equiconsistent with the existence of a measurable cardinal that every unbounded subset of $\Set{\omega_n}{n<\omega}$ is $\mathbf{\Sigma}_1(Ord)$-stationary in $\omega_\omega$ (Theorem \ref{Thm:Measurable}). 

    \item It is equiconsistent with the existence of a Mahlo cardinal that there is a regular cardinal $\mu$ for which the singleton $\{\mu\}$ is $\Sigma_1(\HH{\mu})$-stationary in $\mu^+$ (Theorem \ref{Thm:Mahlo}). 

    \item While it is provable that for every set $A$ of cardinality less than the \emph{reaping number $\mathfrak{r}$}\footnote{See Definition \ref{definition:ReapingNumber} below.} and every singular cardinal $\kappa$ of countable cofinality,  there are disjoint $\Sigma_1(A)$-stationary subsets of $\kappa$ (Proposition \ref{proposition:CountCofBistat}), it is equiconsistent with the existence of a measurable cardinal that there is a singular cardinal $\kappa$ of countable cofinality such that for every subset $A$ of $\HH{\kappa}$ of cardinality $\mathfrak{r}$, there are are disjoint $\Sigma_1(A)$-stationary subsets of $\kappa$ (Corollary \ref{Cor:EquiconsistentReaping}). 
\end{itemize}

 We now briefly outline the structure of this paper:
 In Section \ref{section:FirstResults}, we prove preliminary results about the theory of $\Sigma_n(R)$-stationary sets using mostly combinatorial arguments. In addition, we introduce an auxiliary property of $\Sigma_n$-undefinability of an ordinal, which plays a key role in the result of the paper.
 %
 In Section \ref{section:UndefProp}, we examine $\Sigma_1$-definability in the Dodd-Jensen core model, and prove several results regarding $\Sigma_1$-stationary sets in 
 $\KK^{DJ}$. 
 Section \ref{Section:LargeCardinalAndSigma_1} is devoted to showing that $\Sigma_1$-stationarity is weak in the presence of sufficiently large cardinals (e.g., measurable cardinals or stably measurable cardinals), or at small cardinals with strong partition properties (e.g., when $\omega_\omega$ is J\'{o}nsson). 
 In Section \ref{Section:ConsistencyResults}, we build on the results about $\KK^{DJ}$ as well as on forcing results with large cardinals, to prove several equiconsistency results about of weakness of $\Sigma_1(Ord)$-definable closed unbounded sets. 
 In Section \ref{Section:Concluding Remarks and Open Problems}, we conclude the paper by listing  some  problems left open by our results.


\section{Preliminaries about $\Sigma_n$-Stationary sets}\label{section:FirstResults}

In this section, we start to develop the  theory of $\Sigma_n(R)$-stationary sets for natural classes $R$.


\subsection{A proper hierarchy}

 In order to motivate the below results, we start our investigations of $\Sigma_n(R)$-stationary sets by presenting a setting in which these sets form a properly descending hierarchy in the parameter $n$.

 \begin{theorem}\label{theorem:ProperyHierarchySigmaNStationary}
  Assume that the $\GCH$ holds. 
  Then for every  uncountable regular cardinal  $\kappa$ and every  cardinal $\theta$ satisfying  $\theta=\theta^\kappa$, there is a uniformly definable partial order $\po_{\kappa,\theta}$ such that forcing with $\po_{\kappa,\theta}$ preserves cofinalities and, if $G$ is $\po_{\kappa,\theta}$-generic over $V$, then, in $V[G]$, for every natural number $n>0$ and every set $A$ of cardinality less than $\theta$ with the property that the set $\{A\}$ is definable by a $\Sigma_2$-formula with parameters in $A$,  there is a subset of $\kappa$ that is $\Sigma_n(A)$-stationary in $\kappa$ and not $\Sigma_{n+1}(A)$-stationary in $\kappa$.  
 \end{theorem}

 Note that the  definability assumptions on the parameter set $A$ stated in the above theorem are satisfied in the case where $\theta=\kappa^{++}$ and $A=\HH{\kappa}\cup\kappa^+$. We therefore directly get the following corollary whose statement should be compared with the two settings discussed in Section \ref{section:Intro}, in which  all $\mathbf{\Sigma}_1(\kappa^+)$-stationary subsets of an uncountable cardinal $\kappa$ are stationary.

 \begin{corollary}
     If the $\GCH$ holds, then for every  uncountable regular cardinal  $\kappa$, there is a uniformly definable cofinality preserving partial order $\po_\kappa$ with the property that whenever $G$ is $\po_\kappa$-generic over $V$ and $n>0$ is a natural number, then, in $V[G]$, there is a subset of $\kappa$ that is $\mathbf{\Sigma}_n(\kappa^+)$-stationary in $\kappa$ and not $\mathbf{\Sigma}_{n+1}(\kappa^+)$-stationary in $\kappa$.  \qed 
 \end{corollary}

 The above result is a consequence of the following two observations that might be of independent interest:

 \begin{proposition}\label{proposition:ProperStatHierarchy}
  Let $n>0$ be a natural number, let $\kappa$  be an uncountable cardinal with $\POT{\kappa}\subseteq\HOD$ and let $A$ be a set with the property that the set $\{A\}$ is definable by a $\Sigma_{n+1}$-formula with parameters in $A$. If there is a non-stationary subset of $\kappa$ that is $\Sigma_n(A)$-stationary in $\kappa$, then there is a subset of $\kappa$ that is $\Sigma_n(A)$-stationary in $\kappa$ and not $\Sigma_{n+1}(A)$-stationary in $\kappa$. 
 \end{proposition}

 \begin{proof}
  Let $E$ be the least non-stationary subset of $\kappa$ in the canonical well-ordering of $\HOD$ that is $\Sigma_n(A)$-stationary in $\kappa$. Since the collection of all initial segments of the canonical well-ordering of $\HOD$ is definable by a $\Sigma_2$-formula without parameters and the collection $C^{(n)}$ of all ordinals $\lambda$ with $V_\lambda\prec_{\Sigma_n}V$ is definable by a $\Pi_n$-formula without parameters (see {\cite[Section 1]{zbMATH06029795}}), we know that the set $\{E\}$ is definable by a $\Sigma_{n+1}$-formula with parameters in $A$. Let $C$ denote the least closed unbounded subset of $\kappa$ in the canonical well-ordering of $\HOD$ that is disjoint from $E$. But then we know that the set $\{C\}$ is also definable by a $\Sigma_{n+1}$-formula with parameters in $A$, and therefore $C$ witnesses that $E$ is not $\Sigma_{n+1}(A)$-stationary in $\kappa$. 
 \end{proof}

 \begin{proposition}\label{proposition:CohenNonStationarySigmaNStationary}
  Let $\kappa$ be an uncountable cardinal with $\kappa^{{<}\kappa}=\kappa$, let $\theta>\kappa$ be a cardinal with  $\theta^\kappa=\theta$, let $G$ be $\Add{\kappa}{\theta}$-generic over $\VV$ and let $\calC$ be a collection of closed unbounded subsets of $\kappa$ of cardinality less than $\theta$ in $V[G]$.   Then, in $\VV[G]$, there exists a non-stationary subset $E$ of $\kappa$ with the property that $C\cap E\neq\emptyset$ holds for all $C\in\calC$. 
 \end{proposition}

 \begin{proof}
   By our assumptions, we can find a generic extension $M$ of $\VV$ such that $\calC\in M\subsetneq\VV[G]$ and $\VV[G]$ is a non-trivial $\Add{\kappa}{\theta}$-generic extension of $M$. 
  Let $H\in\VV[G]$ be $\Add{\kappa}{1}$-generic over $M$, let $X$ denote the subset of $\kappa$ corresponding to $H$ and set $E=X\setminus\Lim(X)$. Genericity then ensures that $E$ intersects every  unbounded subset of $\kappa$ in $M$. In particular, we know that $C\cap E\neq\emptyset$ holds for every $C\in\calC$. Finally, since $E$ is disjoint from $\Lim(X)$, we know that $E$ is a non-stationary subset of $\kappa$ in $\VV[G]$.    
 \end{proof}

 A combination of these two observations now directly yields the desired consistency proof:

  \begin{proof}[Proof of Theorem \ref{theorem:ProperyHierarchySigmaNStationary}]
     Assume that the $\GCH$ holds. 
  Given an uncountable regular cardinal  $\kappa$ and a  cardinal $\theta$ satisfying  $\theta=\theta^\kappa$, we define $\po_{\kappa,\theta}$ to be the two-step iteration $\Add{\kappa}{\theta}*\dot{\QQQ}$, where $\dot{\QQQ}$ is the canonical $\Add{\kappa}{\theta}$-name for the ${<}\theta^+$-closed partial order that codes $\POT{\kappa}$ into the $\GCH$-pattern above $\theta^+$ (see, for example, \cite{MR3304634}). Let $G*H$ be ($\Add{\kappa}{\theta}*\dot{\QQQ}$)-generic over $V$. Fix a natural number $n>0$ and a set $A\in V[G,H]$ of cardinality less than $\theta$ with the property that,  in $V[G,H]$, the set $\{A\}$ is definable by a $\Sigma_2$-formula with parameters in $A$. Since $\POT{\POT{\kappa}}^{V[G]}=\POT{\POT{\kappa}}^{V[G,H]}$ and, in $V[G,H]$, there are less than $\theta$-many closed unbounded subsets $C$ of $\kappa$  with the property that the set $\{C\}$ is definable by a $\Sigma_n$-formula with parameters in $A\cup\{\kappa\}$, an application of Proposition \ref{proposition:CohenNonStationarySigmaNStationary} shows that, in $V[G,H]$, there is a non-stationary subset of $\kappa$ that is $\Sigma_n(A)$-stationary in $\kappa$. But our setup ensures that  $\POT{\kappa}^{V[G]}=\POT{\kappa}^{V[G,H]}\subseteq\HOD^{V[G,H]}$ and hence Proposition \ref{proposition:ProperStatHierarchy} allows us to conclude that, in $V[G,H]$, there is a subset of $\kappa$ that is $\Sigma_n(A)$-stationary in $\kappa$ and not $\Sigma_{n+1}(A)$-stationary in $\kappa$.  
 \end{proof}

 In  a sequel to this paper, we will provide analogous results about the properness of the hierarchy of $\Sigma_n(A)$-stationary sets for singular cardinals.


\subsection{Combinatorial arguments}

We now continue by analyzing basic structural features of the collection of all $\Sigma_n(A)$-stationary sets. 
All results presented in this section are purely combinatorial, in the sense that they only rely on counting arguments. We nevertheless decided to phrase them in such a way that they make statements about collections of definable subsets. These formulations motivate several of our later results that will show that, in general, we cannot relax the stated assumptions on the sizes of parameter sets. 
We start our analysis by comparing $\Sigma_n(A)$-stationarity with standard stationarity and isolating settings in which counting arguments ensure the existence of  $\Sigma_n(A)$-stationary sets that are not stationary.

\begin{proposition}\label{proposition:SigmaStatNonStat}
 If $\kappa$ is a cardinal of uncountable cofinality, $A$ is a set of cardinality at most $\cof{\kappa}$ and $n$ is a natural number, then there is an unbounded, non-stationary subset of $\kappa$ that is $\Sigma_n(A)$-stationary in $\kappa$. 
\end{proposition}

\begin{proof}
 Let $\seq{C_\alpha}{\alpha<\cof{\kappa}}$ be an enumeration of all closed unbounded subsets $C$ of $\kappa$ with the property that the set $\{C\}$ is definable by a $\Sigma_n$-formula with parameters in $A\cup\{\kappa\}$. In addition,  define $C$ to be a diagonal intersection of these sets with respect to some strictly increasing continuous sequence of order type $\cof{\kappa}$ in $\kappa$. Set $S=C\setminus\Lim(C)$.  Then $S$ is an unbounded and non-stationary subset of $\kappa$. Moreover, we have $C\cap S\neq\emptyset$ whenever $C$ is a closed unbounded subset  of $\kappa$ with the property that the set $\{C\}$ is definable by a $\Sigma_n$-formula with parameters in $A\cup\{\kappa\}$.  
\end{proof}

 The results of this paper show that the implication in the above proposition can fail if we consider sets of parameters of cardinality  $\cof{\kappa}^+$. In the case of regular cardinals,  the examples given in Section \ref{section:Intro} about $L$ and under $MM$ provide examples of such failures. 
 For singular cardinals $\kappa$ of uncountable cofinality, Corollary \ref{corollary:SingularUncCofStationary} below will show that, in the Dodd-Jensen core model, stationarity coincides with $\Sigma_1(\POT{\cof{\kappa}})$-stationarity, and therefore also with $\mathbf{\Sigma}_1$-stationarity in $\kappa$.
 Moreover, \ref{corollary:StrengthSingularUncountCof}, shows that when $\kappa$ is singular of uncountable cofinality, the existence of 
 $\mathbf{\Sigma}_1$-stationarity subset of $\kappa$ that is not stationary is equi-consistent with the existence of $\cof{\kappa}$ many measurable cardinals.
 In the case of singular cardinals of countable cofinalities, 
 Theorem \ref{Thm:Measurable} will give an analogous equi-consistency result. We start with a quick combinatorial argument.

\begin{proposition}\label{proposition:CofCountableSigmaStatNonStat}
 If $\kappa$ is a singular  cardinal of countable cofinality, $A$ is a set of cardinality less than $\kappa^\omega$ and $n$ is a natural number, then there is an unbounded subset of $\kappa$ that is $\Sigma_n(A)$-stationary in $\kappa$ and whose complement in $\kappa$ is unbounded in $\kappa$. 
\end{proposition}

\begin{proof}
 Since the set $[\kappa]^\omega$ contains an almost disjoint family of cardinality $\kappa^\omega$ that consists of unbounded subsets of $\kappa$, we can find an element $b\in[\kappa]^\omega$ that is unbounded in $\kappa$ and has the property that no infinite subset of $b$ is definable by a $\Sigma_n$-formula with parameters in $A\cup\{\kappa\}$. Then $\kappa\setminus b$ is $\Sigma_n(A)$-stationary in $\kappa$ and the complement of this set in $\kappa$ is unbounded in $\kappa$. 
\end{proof}

 We continue by comparing the structural properties of $\Sigma_n(A)$-stationary sets with those of standard stationary sets. Our first focus in this comparison will be the question of the existence of disjoint $\Sigma_n(A)$-stationary sets at various uncountable  cardinals. 
 In the case of cardinals of  uncountable cofinality, the fact that all stationary sets are $\Sigma_n(A)$-stationary already ensures the existence of such sets. 
 The next result strengthens this conclusion by showing that a version of \emph{Solovay's Theorem} on the splitting of stationary sets holds for $\Sigma_n(A)$-stationary sets.

\begin{proposition}
Suppose that $\kappa$ is a cardinal of uncountable cofinality, $A$ is a set of cardinality $\cof{\kappa}$ with $A\cap\kappa$ cofinal in $\kappa$,  $n < \omega$, and $S$ is a $\Sigma_n(A)$-stationary subset of $\kappa$. Then there exists a partition $\seq{S_\alpha}{\alpha<\cof{\kappa}}$ of $S$ into $\Sigma_n(A)$-stationary subsets. 
\end{proposition}

\begin{proof}
Since $A\cap\kappa$ is unbounded in $\kappa$, we know that there are $\cof{\kappa}$-many $\Sigma_n(A)$-definable closed unbounded subsets of $\kappa$. 
Let $\seq{C_\alpha}{\alpha<\cof{\kappa}}$ be an enumeration of all closed unbounded subsets $C$ of $\kappa$ with the property that the set $\{C\}$ is definable by a $\Sigma_n$-formula with parameters in $A\cup\{\kappa\}$. 
Since the assumption that $A\cap\kappa$ is unbounded in $\kappa$ implies that $C_\alpha\cap S$ is unbounded in $\kappa$ for all $\alpha<\cof{\kappa}$, we can find a strictly increasing sequence $\seq{\sigma_\alpha}{\alpha<\cof{\kappa}}$ of elements of $S$ with the property that $\sigma_{\goedel{\alpha_0}{\alpha_1}}\in C_{\alpha_1}$ for all $\alpha_0,\alpha_1<\cof{\kappa}$, where $\map{\goedel{\cdot}{\cdot}}{Ord\times Ord}{Ord}$ denotes the G\"odel pairing function. If we now pick a partition $\seq{S_\alpha}{\alpha<\cof{\kappa}}$ of $S$ with the property that $\Set{\sigma_{\alpha,\beta}}{\beta<\cof{\kappa}}\subseteq S_\alpha$ holds for all $\alpha<\cof{\kappa}$, then  each $S_\alpha$ is $\Sigma_n(A)$-stationary in $\kappa$. 
\end{proof}

In the case of cardinals of countable cofinality,  the  existence of disjoint $\Sigma_n(A)$-stationary sets  turns out to be closely connected to the \emph{reaping number $\mathfrak{r}$}. 

\begin{definition}\label{definition:ReapingNumber}
$\mathfrak{r}$ is the least cardinality of a subset $A$ of $[\omega]^\omega$ with the property that for every $b\in[\omega]^\omega$, there is $a\in A$ such that either $a\setminus b$ or $a\cap b$ is finite. 
\end{definition}

We will later show that, in general, the conclusion of the following proposition cannot be extended to sets of parameters of cardinality $\mathfrak{r}$ (see Corollaries \ref{corollary:NoMeasurableCorollaries3} and \ref {Cor:EquiconsistentReaping} below).

\begin{proposition}\label{proposition:CountCofBistat}
 Let $\kappa$ be a singular cardinal of countable cofinality, let $A$ be a set of cardinality less than $\mathfrak{r}$ and let $n$ be a natural number. Then there exists a subset $E$ of $\kappa$ with the property that both $E$ and $\kappa\setminus E$ are $\Sigma_n(A)$-stationary in $\kappa$. 
\end{proposition}

The proof of this proposition relies on the equivalence provided by the next lemma:

\begin{lemma}\label{lemma:reaping}
 The following statements are equivalent for every infinite cardinal $\theta$: 
 \begin{enumerate}
  \item\label{item:reaping1} $\theta<\mathfrak{r}$. 

  \item\label{item:reaping2} For every singular cardinal $\kappa$ of countable cofinality and every subset $A$ of $[\kappa]^\omega$ that consists of cofinal subsets of $\kappa$ and has cardinality at most $\theta$, there exists a subset $E$ of $\kappa$ such that for every $a\in A$, both $a\cap E$ and $a\setminus E$ are infinite. 

  \item\label{item:reaping3} There is a singular cardinal $\kappa$ of countable cofinality with the property that for every subset $A$ of $[\kappa]^\omega$ that consists of cofinal subsets of $\kappa$ and has cardinality at most $\theta$, there exists a subset $E$ of $\kappa$ such that for every $a\in A$, both $a\cap E$ and $a\setminus E$ are infinite. 
 \end{enumerate}
\end{lemma}

\begin{proof}
 First, assume that \eqref{item:reaping1} holds and \eqref{item:reaping2} fails. Then there is a singular cardinal $\kappa$ of countable cofinality and a subset $A$ of $[\kappa]^\omega$ such that $A$ consists of cofinal subsets of $\kappa$, $\betrag{A}\leq\theta$ and for every subset $E$ of $\kappa$, there is $a\in A$ such that either $a\cap E$ or $a\setminus E$ is finite. 
 Let $\seq{\kappa_n}{n<\omega}$ be a strictly increasing sequence that is cofinal in $\kappa$ with $\kappa_0=0$. Given $a\in A$, define $$b_a ~ = ~ \Set{n<\omega}{a\cap[\kappa_n,\kappa_{n+1})\neq\emptyset} ~ \in ~ [\omega]^\omega.$$ Since $\betrag{A}<\mathfrak{r}$, we can now find $c\in[\omega]^\omega$ with the property that for all $a\in A$, both $b_a\setminus c$ and $b_a\cap c$ are infinite. Set $$E ~ = ~ \bigcup_{n\in c}[\kappa_n,\kappa_{n+1}).$$ By our assumptions, there is $a\in A$ with the property that either $a\setminus E$ or $a\cap E$ is finite. But this implies that either $b_a\setminus c$ or $b_a\cap c$ is finite, a contradiction.

 Now, assume that \eqref{item:reaping3} holds and \eqref{item:reaping1} fails. 
 Then there is a singular cardinal $\kappa$ of countable cofinality with the property that for every subset $A$ of $[\kappa]^\omega$ that consists of cofinal subsets of $\kappa$ and has cardinality at most $\theta$, there exists a subset $E$ of $\kappa$ such that for every $a\in A$, both $a\cap E$ and $a\setminus E$ are infinite. 
 Fix $A\subseteq[\omega]^\omega$ of cardinality $\mathfrak{r}$ such that for every $b\in[\omega]^\omega$, there exists $a\in A$ with the property that either $a\setminus b$ or $a\cap b$ is finite. 
 Pick a s strictly increasing sequence $\seq{\kappa_n}{n<\omega}$ that is cofinal in $\kappa$ and, given $a\in A$, define $b_a=\Set{\kappa_n}{n\in a}\in[\kappa]^\omega$. Since $b_a$ is cofinal in $\kappa$ for all $a\in A$, we can now find a subset $E$ of $\kappa$ such that for every $a\in A$, both $b_a\cap E$ and $b_a\setminus E$ are infinite. Set $c=\Set{n<\omega}{\kappa_n\in E}\in[\omega]^\omega$. If $a\in A$, then both $a\setminus c$ and $a\cap c$ are infinite, contradicting our assumptions on $A$.  
\end{proof}

\begin{proof}[Proof of Proposition \ref{proposition:CountCofBistat}]
  Let $\kappa$ be a singular cardinal of countable cofinality, let $A$ be a set of cardinality less than $\mathfrak{r}$ and let $n>0$ be a natural number. Then there exists a subset $A'$ of $[\kappa]^\omega$ of cardinality less than $\mathfrak{r}$ that consists of cofinal sequences and has the property that for every closed unbounded subset $C$ of $\kappa$ such that the set $\{C\}$ is definable by a $\Sigma_n$-formula with parameters in $A\cup\{\kappa\}$, there exists $a\in A'$ with $a\subseteq C$. 
  Using Lemma \ref{lemma:reaping}, we can now find a subset $E$ of $\kappa$ such that for every $a\in A$, both $a\setminus E$ and $a\cap E$ are infinite. Then both $E$ and $\kappa\setminus E$ are $\Sigma_n(A)$-stationary in $\kappa$. 
\end{proof}

We close this section by comparing another aspect of the behavior of stationary sets with its counterpart in the definable context. While the collection of all closed unbounded subsets of a cardinal of uncountable cofinality is closed under intersections and therefore all of these subsets are stationary, these implications can obviously fail at a singular cardinal of countable cofinality, where all unbounded sets of order type $\omega$ are closed unbounded and stationarity coincides with coboundedness. The following lemma 
 completely characterizes the settings in which these  implications also hold in the definable context:

\begin{lemma}\label{lemma:CharClubsClosed}
 Given a class $A$ and a natural number $n>0$, the following statements are equivalent for every singular cardinal $\kappa$ of countable cofinality: 
 \begin{enumerate}
     \item There is a cofinal function $\map{c}{\omega}{\kappa}$ that is definable by a $\Sigma_n$-formula with parameters in $A\cup\{\kappa\}$. 

     \item There are disjoint closed unbounded subsets $C_0$ and $C_1$ of $\kappa$ with the property that the sets $\{C_0\}$ and $\{C_1\}$ are both definable by $\Sigma_n$-formulas with parameters in $A\cup\{\kappa\}$. 
     
     \item There are closed unbounded subsets $C_0$ and $C_1$ of $\kappa$ with $C_0\cap C_1$ bounded in $\kappa$ and the property that the sets $\{C_0\}$ and $\{C_1\}$ are both definable by $\Sigma_n$-formulas with parameters in $A\cup\{\kappa\}$. 
 \end{enumerate}
\end{lemma}

\begin{proof}
  First, we assume that there is a cofinal function $\map{c}{\omega}{\kappa}$ that is definable by a $\Sigma_n$-formula with parameters in $A\cup\{\kappa\}$. Define $C_0=\Set{c(n)+\omega}{n<\omega}$ and $C_1=\Set{c(n)+\omega+1}{n<\omega}$. Then $C_0$ and $C_1$ are disjoint closed unbounded in $\kappa$, and the sets $\{C_0\}$ and $\{C_1\}$ are definable by $\Sigma_n$-formulas with parameters in $A\cup\{\kappa\}$. 
 In the other direction, assume that there are closed unbounded subsets $C_0$ and $C_1$ of $\kappa$ with $C_0\cap C_1$ bounded in $\kappa$ and the property that the sets $\{C_0\}$ and $\{C_1\}$ are both definable by $\Sigma_n$-formulas with parameters in $A\cup\{\kappa\}$. We define $\mu=\max(C_0\cap C_1)<\kappa$ and let $\map{c}{\omega}{\kappa}$ denote the unique function with  $c(0)=\mu$ and $$c(2k+2-i) ~ = ~ \min(C_i\setminus(c(2k+1-i)+1))$$ for all $k<\omega$ and $i<2$. The $\Sigma_n$-Recursion Theorem then implies that $c$ is definable by a $\Sigma_n$-formula with parameters in $A\cup\{\kappa\}$. Moreover, we know that $c$ is cofinal in $\kappa$, because otherwise $\mu<\sup_{k<\omega}c(k)\in C_0\cap C_1$. 
\end{proof}

Using an argument similar to the one of the previous Lemma, we obtain the following corollary, which will later allow us to show that definable closed unbounded sets behave nicely in  various settings.

\begin{corollary}\label{corollary:CriterionClosedIntersections}
 Let $\kappa$ be a singular cardinal of countable cofinality, let $A$ be a class and let $n>0$ be a natural number with the property that there exists a subset of $\kappa$ that is $\Sigma_n(A)$-stationary in $\kappa$ and consists of cardinals. Then the collection of all closed unbounded subsets $C$ of $\kappa$ with the property that the set $\{C\}$ is definable by a $\Sigma_n$-formula with parameters in $A\cup\{\kappa\}$ is closed under intersections. \qed
\end{corollary}


\subsection{$\Sigma_n$-undefinability property}

We introduce a notion that will allow us to show that various nonstationary sets of cardinals $E \subseteq \kappa$  of an uncountable cardinal $\kappa$ are $\mathbf{\Sigma}_n$-stationary.

\begin{definition}\label{definition:UnDefProp}
 Given uncountable cardinals $\mu<\kappa$, an ordinal $\gamma\geq\kappa$ and a natural number $n$, we say that the cardinal $\kappa$ has the \emph{$\Sigma_n(\mu,\gamma)$-undefinability property} if no ordinal $\alpha$ in the interval $[\mu,\kappa)$ has the property that the set $\{\alpha\}$ is definable by a $\Sigma_n$-formula with parameters in the set $\HH{\mu}\cup\{\kappa,\gamma\}$. 
 Moreover, we say that $\kappa$ has the \emph{$\Sigma_n(\mu)$-undefinability property} if it has the $\Sigma_n(\mu,\kappa)$-undefinability property. 
\end{definition}

The next lemma shows how this undefinability property is connected to $\Sigma_n$-stationarity:

\begin{lemma}\label{lemma:UnDefStationary}
 Given uncountable cardinals $\mu<\kappa$, an ordinal $\gamma\geq\kappa$ and a natural number $n>0$, if the cardinal $\kappa$ has the $\Sigma_n(\mu,\gamma)$-undefinability property, then the set $\{\mu\}$ is $\Sigma_n(\HH{\mu}\cup\{\gamma\})$-stationary in $\kappa$.  
\end{lemma}

\begin{proof}
 Let $C$ be a closed unbounded subset of $\kappa$ with the property that the set $\{C\}$ is definable  by a $\Sigma_n$-formula with parameters in $\HH{\mu}\cup\{\kappa,\gamma\}$. 
 Assume, towards a contradiction, that $\mu$ is not an element of $C$. Set $\nu=\min(C\setminus\mu)>\mu$. Then  $C\cap\mu\neq\emptyset$, because otherwise $\nu=\min(C)$ and this would imply that the set  $\{\nu\}$ is definable by a $\Sigma_n$-formula with parameters in $\HH{\mu}\cup\{\kappa,\gamma\}$. This allows us to define $\rho=\max(C\cap\mu)<\mu$. But then $\nu=\min(C\setminus(\rho+1))$ and hence $\{\nu\}$ is definable by a $\Sigma_n$-formula with parameters in $\HH{\mu}\cup\{\kappa,\gamma\}$, a contradiction.  
\end{proof}

The following direct corollary of the above lemma will later allow us to isolate various examples of $\mathbf{\Sigma}_1$-stationary subsets of cardinals of countable cofinality that are not stationary (i.e., not cobounded) in the given limit cardinal.

\begin{corollary}\label{corollary:UnbCardStat}
 Let $\kappa$ be a limit cardinal, let $\gamma\geq\kappa$ be an ordinal, let $n>0$ be a natural number and let $E$ be the set of uncountable cardinals $\mu<\kappa$ with the property that $\kappa$ has the $\Sigma_n(\mu,\gamma)$-undefinability property for all $\mu \in E$. If $E$ is unbounded in $\kappa$, then $E$ is $\mathbf{\Sigma}_n(\{\gamma\})$-stationary in $\kappa$. \qed
\end{corollary}


\section{Undefinability in the Dodd-Jensen core model}\label{section:UndefProp}

 In the following, we establish basic definability and undefinability results dealing with \emph{Dodd-Jensen core model $\KK^{DJ}$}. These results make use of the presentation of this model in \cite{MR730856} and \cite{MR926749}. 
 In this setting, a set $M$ is a \emph{premouse at an ordinal $\mu>\omega$} if $M$ is of the form $J^U_\eta$ and $$\langle M,\in,U\rangle\models\anf{\text{$U$ is a normal 
 ultrafilter on $\mu$}}.$$ 
 Moreover, if $M$ is a premouse at $\mu$, then we define the \emph{lower part of $M$} to be the set $lp(M)=M\cap V_\mu$. 
 Given a premouse $M$ and an ordinal $\delta$, an iteration of $M$ of length $\delta$ is given by a sequence $\seq{M_\alpha}{\alpha<\delta}$ of premice and a commuting system $\seq{\map{j_{\alpha,\beta}}{M_\alpha}{M_\beta}}{\alpha\leq\beta<\delta}$ of $\Sigma_1$-elementary embeddings such that the following statements hold:
 \begin{itemize}
     \item $M=M_0$ and $j_{\alpha,\alpha}=\id_{M_\alpha}$ for all $\alpha<\delta$. 

     \item If $\alpha+1<\delta$, $M_\alpha=J_\eta^U$ is a premouse at $\mu$ and $M_{\alpha+1}=J_\zeta^W$ is a premouse at $\nu$, then $M_{\alpha+1}$ is the transitive collapse of the ultrapower of $M_\alpha$ using $U$, $j_{\alpha,\alpha+1}$ is the corresponding ultrapower embedding, $j_{\alpha,\alpha+1}(\mu)=\nu$ and $$W ~ = ~ \Set{[f]_U}{f\in {}^\mu M_\alpha\cap M_\alpha, \Set{\xi<\mu}{f(\xi)\in U}\in U}.$$

     \item If $\gamma\in\Lim\cap\delta$, then $\langle M_\gamma,\seq{\map{j_{\alpha,\gamma}}{M_\alpha}{M_\gamma}}{\alpha<\gamma}\rangle$ is a direct limit of $\langle\seq{M_\alpha}{\alpha<\gamma},\seq{\map{j_{\alpha,\beta}}{M_\alpha}{M_\beta}}{\alpha\leq\beta<\gamma}\rangle$. 
 \end{itemize}
 If such an iterations exists, then it is uniquely determined and it is called the $\delta$-iteration of $M$.  
 A premouse $M$ is then called a \emph{mouse}, if $\delta$-iterations of $M$ exist   for all $\delta\in\On$. Note that, since the iterability of a premouse can be checked in every transitive structure of uncountable ordinal height that contains the mouse and satisfies a sufficiently strong fragment of $\ZFC$  (see {\cite[Theorem 2.7]{MR926749}}), it follows that the class of all mice is $\Sigma_1$-definable from every uncountable ordinal. 
 We can now use such iterations to compare a mouse $M=J^U_\eta$ at some ordinal $\mu$ with a mouse $N=J^W_\zeta$ at an ordinal $\nu$, in the sense that they allow us to find mice $M^\prime=J^C_{\eta^\prime}$ and $N^\prime=J^C_{\zeta^\prime}$ at the same ordinal $\rho$ and $\Sigma_1$-elementary embeddings $\map{j}{M}{M^\prime}$ and $\map{i}{N}{N^\prime}$ with $j(\mu)=i(\nu)=\rho$ (see {\cite[Lemma 1.13]{MR730856}}). 
  In our arguments below, we will frequently make use of results of Dodd and Jensen (see {\cite[pp. 238-241]{MR730856}}) that show that the Dodd-Jensen core model $\KK^{DJ}$ is equal to the union of the constructible universe $L$ and the lower parts $lp(M)$ of all mice $M$.

  We now prove three lemmata that show that various objects that witness the accessibility of cardinals in $\KK^{DJ}$ are simply definable.

\begin{lemma}\label{lemma:InitialSegmentsKDJ}
 If $\kappa$ is an infinite cardinal that is not inaccessible in $\KK^{DJ}$, then the set $\{\HH{\kappa}^{\KK^{DJ}}\}$ is definable by a $\Sigma_1$-formula with parameter $\kappa$. 
\end{lemma}

\begin{proof}
 Since the $\GCH$ holds in $\KK^{DJ}$ and $\HH{\aleph_0}^{\KK^{DJ}}=\LL_\omega$, we may assume that either  $\kappa$ is a successor cardinal in $\KK^{DJ}$ or $\kappa$ is singular  in $\KK^{DJ}$. Moreover, we may assume that there exists a mouse, because otherwise $\KK^{DJ}=\LL$ and $\{\HH{\kappa}^{\KK^{DJ}}\}=\{\LL_\kappa\}$ is $\Sigma_1$-definable in the desired way.

 \begin{claim*}
  If $M$ is a mouse at some $\nu>\kappa$ such that $\kappa$ is either a successor cardinal or a singular cardinal in $M$, then $\HH{\kappa}^{\KK^{DJ}}\subseteq M$. 
 \end{claim*}

 \begin{proof}[Proof of the Claim]
  Fix $x\in\HH{\kappa}^{\KK^{DJ}}$. Then there exists a mouse $N_0$ at some cardinal $\mu<\kappa$ such that $\betrag{N_0}<\kappa$ and $x\in lp(N_0)$. 
 Let $N$ be the mouse obtained by iterating the top measure of $N_0$ $\kappa$ many times. Then clearly, $x \in lp(N)$ and $\kappa$ is the critical point of the top measure of $N$, and is therefore inaccessible in $N$. The coiteration of $M,N$ which can only involve the top measures of the two mice, results in $\Sigma_1$-elementary embeddings $\map{\pi_M}{M}{M^\prime}$ and $\map{\pi_N}{N}{N^\prime}$, with the following properties:
 \begin{itemize}
     \item $\kappa$ is inaccessible in $N^{\prime}$ but not in $M^{\prime}$,
     \item $x \in lp(N^{\prime})$, 
     \item $\power(\kappa)^M = \power(\kappa)^{M'}$, and
     \item 
 one of $M^\prime$,  $N^\prime$ is an initial segment of the other. 
 \end{itemize}
 It is therefore clear that $N^{\prime}$
 must be an initial segment of $M^{\prime}$, from which we conclude  that $x \in \power^{M'}(\kappa)$ and therefore that $x \in M$. 
 \end{proof}

 As outlined earlier, the assumption that $\KK^{DJ}\neq\LL$ implies that $\KK^{DJ}$ is equal to the union of all lower parts of mice. In particular, there exists a mouse $M$ at an ordinal above $\kappa$ with the property that $\kappa$ is either a successor cardinal or a singular cardinal in $M$. By the above claim, we know that $\HH{\kappa}^{\KK^{DJ}}=\HH{\kappa}^M$ holds for every mouse $M$ with these properties. By combining this implication with earlier remarks about the definability of the class of all mice, we can derive the statement of the lemma. 
\end{proof}

\begin{lemma}\label{lemma:CofInKDJ}
  If $\kappa$ is an infinite cardinal, then the set $\{\cof{\kappa}^{\KK^{DJ}}\}$ is definable by a $\Sigma_1$-formula with parameter $\kappa$. 
\end{lemma}

\begin{proof}
 This is trivial in the case $\kappa$ is not is singular in $\KK^{DJ}$. Assuming it is singular, we know that $\cof{\kappa}^{\KK^{DJ}}$ is the unique ordinal $\xi<\kappa$ with the property that $\xi$ is regular in $\HH{\kappa}^{\KK^{DJ}}$ and $\KK^{DJ}$ contains a cofinal function $\map{c}{\xi}{\kappa}$. 
 Since the class $\KK^{DJ}$ is definable by a $\Sigma_1$-formula with parameter $\kappa$ (see, for example, the proof of {\cite[Lemma 4.13]{Sigma1Partitions}}), we can apply Lemma \ref{lemma:InitialSegmentsKDJ} to conclude that the set $\{\cof{\kappa}^{\KK^{DJ}}\}$ is also definable by a $\Sigma_1$-formula with parameter $\kappa$.  
\end{proof}

\begin{lemma}\label{lemma:LeastCofinal}
 If $\kappa$ is an infinite cardinal and $c$ is the $<_{\KK^{DJ}}$-least cofinal function from $\cof{\kappa}^{\KK^{DJ}}$ to  $\kappa$ in $\KK^{DJ}$, then the set $\{c\}$ is definable by a $\Sigma_1$-formula with parameter $\kappa$. 
\end{lemma}

\begin{proof}
This is an immediate consequence of the proof of {\cite[Lemma 2.3]{MR3845129}}, which shows that the collection of initial segments of the restriction of $<_{\KK^{DJ}}$ to $\POT{\kappa}$ is definable by a $\Sigma_1$-formula with parameter $\kappa$, and  Lemma \ref{lemma:CofInKDJ} above.
\end{proof}

We now use the above results to show that, in the case of singular cardinals $\kappa$ of uncountable cardinality, the statement of Proposition \ref{proposition:SigmaStatNonStat} cannot be strengthened in $\ZFC$. We will later improve this result to obtain $\cof{\kappa}$-many measurable cardinals from the existence of a singular cardinal $\kappa$ of uncountable cofinality with the property that there exists a non-stationary $\mathbf{\Sigma}_1$-stationary subset of $\kappa$ (see Theorem \ref{theorem:UncCofManyMeasurables} and Corollary \ref{corollary:StrengthSingularUncountCof} below).

\begin{corollary}\label{corollary:SingularUncCofStationary}
    Assume that there is no inner model with a measurable cardinal. If $\kappa$ is a singular cardinal of uncountable cofinality and $S$ is  $\Sigma_1(\POT{\cof{\kappa}^{\KK^{DJ}}})$-stationary in $\kappa$, then $S$ is a stationary subset of $\kappa$. 
\end{corollary}

\begin{proof}
 By our assumption, the results of \cite{MR661475} ensure that $\kappa$ is a singular cardinal in $\KK^{DJ}$. Using Lemma \ref{lemma:LeastCofinal}, we find a closed unbounded subset $C_0$ of $\kappa$ of order-type $\cof{\kappa}^{\KK^{DJ}}$ with the property that both the set $\{C_0\}$ and the monotone enumeration of $C_0$ are definable by a $\Sigma_1$-formula with parameter $\kappa$. Given an arbitrary closed unbounded subset $C$ of $\kappa$, we then know that the intersection $C\cap C_0$ is a closed unbounded subset and the set $\{C\cap C_0\}$ is definable by a $\Sigma_1$-formula with parameters in $\POT{\cof{\kappa}^{\KK^{DJ}}}\cup\{\kappa\}$. This shows that every subset of $\kappa$ that is $\Sigma_1(\POT{\cof{\kappa}^{\KK^{DJ}}})$-stationary in $\kappa$ is stationary in $\kappa$. 
\end{proof}

We now derive further consequences of the above lemmata. The results of the subsequent sections will show that it is possible to use large cardinals to obtain  singular cardinals $\kappa$ where the negations of all  of the listed statements hold.

\begin{corollary}\label{corollary:ConsequencesSingularNoMeasurable}
 Assume that there is no inner model with a measurable cardinal. If $\kappa$ is a singular cardinal, then the following statements hold: 
 \begin{enumerate}
  \item If $\alpha<\kappa$, then the set $\{\alpha\}$ is not $\Sigma_1(\alpha)$-stationary in $\kappa$. 

  \item There is an unbounded subset of $\kappa$ that consists of cardinals and is not $\Sigma_1$-stationary. 

  \item There exists a regressive function $\map{r}{\kappa}{\kappa}$ that is definable by a $\Sigma_1$-formula with parameter $\kappa$ and is not constant on any unbounded subset of $\kappa$. 
 \end{enumerate}
\end{corollary}

\begin{proof}
 Set $\lambda=\cof{\kappa}^{\KK^{DJ}}$ and let $\map{c}{\lambda}{\kappa}$ denote the the $<_{\KK^{DJ}}$-least cofinal function
from $\lambda$ to $\kappa$ in $\KK^{DJ}$. By our assumption, the results of \cite{MR661475} imply that $\lambda<\kappa$ and we can use Lemma \ref{lemma:InitialSegmentsKDJ} to show that the set $\HH{\kappa}^{\KK^{DJ}}$ is definable by a $\Sigma_1$-formula with parameter $\kappa$. Moreover, Lemma \ref{lemma:CofInKDJ} and Lemma \ref{lemma:LeastCofinal} ensure that both the set $\{\lambda\}$ and the function $c$ are definable by a $\Sigma_1$-formulas with parameter $\kappa$. 

 Now, fix $\alpha<\kappa$. If $\alpha\leq\lambda$, then $C=(\lambda,\kappa)$ is a closed unbounded subset of $\kappa$ that is disjoint from $\{\alpha\}$ and has the property that the set $\{C\}$ is definable by a $\Sigma_1$-formula with parameter $\kappa$. In the other case, if $\lambda<\alpha$, then there is $\xi<\lambda$ with $c(\xi)>\alpha$ and $C=(c(\xi),\kappa)$ is a closed unbounded subset of $\kappa$ disjoint from $\{\alpha\}$ with the property that the set $\{C\}$ is definable by a $\Sigma_1$-formula with parameters in $\alpha\cup\{\kappa\}$. 

 Next, assume that $\kappa$ is a limit of limit cardinals in $\KK^{DJ}$ and define $C$ to be the closed unbounded set of all ordinals $\rho<\kappa$ with the property that, in $\KK^{DJ}$, the ordinal $\rho$ is a limit cardinal. Since the set $\{\HH{\kappa}^{\KK^{DJ}}\}$ is definable by a $\Sigma_1$-formula with parameter $\kappa$, it follows that set $\{C\}$ is definable in the same way. 
 Now, let $E$ denote the set of all  successor cardinals of singular cardinals smaller than $\kappa$. Since out setup and the results of \cite{MR661475} ensure that all singular cardinals are singular in $\KK^{DJ}$ and $\KK^{DJ}$ computes the successors of these cardinals correctly, we know that each element of $E$ is the  successor of a singular cardinal in $\KK^{DJ}$ and this shows that $C\cap E=\emptyset$. In particular, we can conclude that $E$ is an unbounded subset of $\kappa$ that consists of cardinals and is not $\Sigma_1$-stationary.

 We now assume that $\kappa$ is not a limit of limit cardinals in $\KK^{DJ}$.  Let $\eta$ denote the least ordinal below $\kappa$ with the property that the intverval $(\eta,\kappa)$ contains no ordinals that are limit cardinals in $\KK^{DJ}$ and we define $C$ to be the set of ordinals in $(\eta,\kappa)$ that are successor ordinals of ordinals that are  cardinals in $\KK^{DJ}$. Our assumptions then imply that $C$ is a cofinal subset of $\kappa$ of order-type $\omega$ and this implies that $C$ is a closed unbounded subset of $\kappa$. Moreover, the fact that the set $\{\HH{\kappa}^{\KK^{DJ}}\}$ is definable by a $\Sigma_1$-formula with parameter $\kappa$ ensures that the set $\{C\}$ is definable in the same way. If we now define $E$ to be the set of all cardinals in the interval $(\eta,\kappa)$, then $E$ is unbounded in $\kappa$ and, since  $C\cap E=\emptyset$ holds, we know that $E$ is not $\Sigma_1$-stationary.

 Finally, define $\map{r}{\kappa}{\lambda}$ to be the unique map with $c(\alpha)=0$ for all $\alpha<\lambda$ and $c(\alpha)=\min\Set{\xi<\lambda}{c(\xi)\geq\alpha}$ for all $\lambda\leq\alpha<\kappa$. Then $r$ is regressive and it is not constant on any unbounded subset of $\kappa$. Moreover, our earlier observations show that $r$ is definable by a $\Sigma_1$-formula with parameter $\kappa$.  
\end{proof}

\begin{corollary}\label{corollary:NoMeasurableCorollaries3}
 Assume that there is no inner model with a measurable cardinal. If $\kappa$ is a singular cardinal of countable cofinality, then the following statements hold: 
 \begin{enumerate}
     \item There are disjoint closed unbounded subsets $C_0$ and $C_1$ of $\kappa$ with the property that the sets $\{C_0\}$ and $\{C_1\}$ are both definable by $\Sigma_1$-formulas with parameters in $\HH{\kappa}\cup\{\kappa\}$. 

     \item There exists a subset $A$ of $\HH{\kappa}$ of cardinality $\mathfrak{r}$ such that every subset of $\kappa$ that is $\Sigma_1(A)$-stationary in $\kappa$ contains a closed unbounded subset $C$ of $\kappa$ with the property that the set $\{C\}$ is definable by a $\Sigma_1$-formula with parameters in $A\cup\{\kappa\}$. 
 \end{enumerate}
\end{corollary}

\begin{proof}
 Set $\lambda=\cof{\kappa}^{\KK^DJ}$.  Our assumptions then imply that $\lambda<\kappa$. Let $\map{c}{\lambda}{\kappa}$ denote the $<_{\KK^{DJ}}$-least cofinal function in $\KK^{DJ}$. Since $\cof{\lambda}=\omega$, we can also fix a cofinal function $\map{d}{\omega}{\lambda}$. 
 Define $C_0=\Set{\omega\cdot((c\circ d)(i))}{i<\omega}$ and $C_1=\Set{\alpha+1}{\alpha\in C_0}$. Then the sets $C_0$ and $C_1$ are disjoint closed unbounded subsets of $\kappa$. Moreover, Lemma \ref{lemma:LeastCofinal} ensures that the sets $\{C_0\}$ and $\{C_1\}$ are both definable by $\Sigma_1$-formulas with parameters $\kappa$ and $d$.  
 Finally, pick a subset $A$ of $\HH{\kappa}$ of cardinality $\mathfrak{r}$ such that $\omega\cup\{d\}\subseteq A$, $A\cap\kappa$ is cofinal in $\kappa$ and for every $b\in[\omega]^\omega$, there is $a\in A\cap[\omega]^\omega$ with the property that either $a\setminus b$ or $a\cap b$ is finite. 
 Let $S$ be a subset of $\kappa$ that is $\Sigma_1(A)$-stationary in $\kappa$. The fact that $A\cap\kappa$ is unbounded in $\kappa$ then implies that the set $b=\Set{i<\omega}{(c\circ d)(i)\in S}$ is infinite. Hence, there exists $a\in A$ with the property that either $a\setminus b$ or $a\cap b$ is finite. Set $C=\Set{(c\circ d)(i)}{i\in a}$. Then $C$ is closed unbounded in $\kappa$ and the set $\{C\}$ is definable by a $\Sigma_1$-formula with parameters in $A\cup\{\kappa\}$. Hence, we know that $C\cap S$ is unbounded in $\kappa$ and this shows that $a\cap b$ is infinite. We can now find $k<\omega$ with  $a\setminus b\subseteq k$ and, if we define $D=\Set{(c\circ d)(i)}{k\leq i\in A}$, then $D$ is a closed unbounded subset of $S$ and the set $\{D\}$ is definable by a $\Sigma_1$-formula with parameters in $A\cup\{\kappa\}$.  
\end{proof}

We end this section with another results about the definability of initial segments of the Dodd-Jensen core model that will be used in our characterizations of \emph{stably measurable} cardinals below.

\begin{lemma}\label{lemma:DefineFromBistationary}
 Let $\kappa$ be an uncountable cardinal and let $E\in \KK^{DJ}$ be a subset of $\kappa$. If $\kappa$ is regular in $\KK^{DJ}$ and $E$ is a bistationary subset of $\kappa$ in $\KK^{DJ}$, then the set $\{\HH{\kappa}^{\KK^{DJ}}\}$ is definable by a $\Sigma_1$-formula with parameter $E$. 
\end{lemma}

\begin{proof}
 Since the class $\KK^{DJ}$ is definable by a $\Sigma_1$-formula with parameter $\kappa$ in $V$, we may assume that $V=\KK^{DJ}$ holds.   
 Moreover, we may assume that there exists a mouse, because otherwise we have $\KK^{DJ}=\LL$ and $\{\HH{\kappa}^{\KK^{DJ}}\}=\{\LL_\kappa\}$ is definable by a $\Sigma_1$-formula with parameter $E$. 2
 The desired statement  is then a direct consequence of the following observation: 

 \begin{claim*}
  If $M$ is a mouse with $E\in lp(M)$, then $\HH{\kappa}\subseteq M$. 
 \end{claim*}

 \begin{proof}[Proof of the Claim]
  Assume, towards a contradiction, that there is an $x\in\HH{\kappa}\setminus M$. Then there is a mouse $N_0$  with  $\betrag{N_0}<\kappa$ and $x\in lp(N_0)$. As above, this allows us to find a mouse $N$ at $\kappa$ with $x\in lp(N)$. Let $\map{\pi_0}{M}{M^\prime}$ and $\map{\pi_1}{N}{N^\prime}$ be the $\Sigma_1$-elementary embeddings obtained by coiterating $M$ and $N$, i.e., either $M^\prime$ is an initial segment of $N^\prime$ or $N^\prime$ is an initial segment of $M^\prime$. 
  Since $\HH{\kappa}^M=\HH{\kappa}^{M^\prime}$ and $x\in lp(N)\setminus M$, we now know that $M^\prime$ is an initial segment of $N^\prime$. This implies that $E\in\POT{\kappa}^M=\POT{\kappa}^{M^\prime}\subseteq\POT{\kappa}^{N^\prime}=\POT{\kappa}^N$. Since the given $N$-ultrafilter on $\kappa$ is equal to the restriction of the closed unbounded filter on $\kappa$ to $\POT{\kappa}^N$, we can conclude that $E$ either contains a closed unbounded subset of $\kappa$ or is disjoint from such a subset. This contradicts the bistationarity of $E$.  
 \end{proof}

  This claim now shows that $\HH{\kappa}$ is the unique set $B$ with the property that there exists a mouse $M$ with $E\in lp(M)$ and $B=\HH{\kappa}^M$. This directly yields the desired $\Sigma_1$-definition of $\{\HH{\kappa}\}$.  
\end{proof}


\section{Large cardinals and $\Sigma_1$-stationary sets}\label{Section:LargeCardinalAndSigma_1}

In Section \ref{section:FirstResults}, we already gave examples of two important features of $\Sigma_1(A)$-stationary sets. First, the collection of these sets can be substantially larger than the collection of ordinary stationary sets. Second, this collection can possess structural features that resemble the behavior of stationary sets. We start by proving results for large cardinals and then extend these results to limits of large cardinals (not necessarily regular). Finally, we show that it is possible to derive similar consequences from Ramsey-theoretic properties that may hold on smaller cardinal.


\subsection{Stably measurable cardinals}

 %
 The following large cardinal property, introduced by Welch in \cite{zbMATH07415211}, turns out to be closely connected to $\Sigma_1$-undefinability considerations.

 \begin{definition}[Welch]
 An uncountable regular cardinal $\kappa$ is \emph{stably measurable} if there exists are 
 \begin{itemize}
 \item transitive set $M$ with $\HH{\kappa}\cup\{\kappa\}\subseteq M\prec_{\Sigma_1}\HH{\kappa^+}$, 
 \item a transitive set $N$ with $M\cup{}^{{<}\kappa}N\subseteq N$, and 
 \item a normal, weakly amenable $N$-ultrafilter $F$ on $\kappa$ with the property that $\langle N,\in,F\rangle$ is iterable.  
 \end{itemize}

\end{definition}

 In \cite{MR2817562}, Sharpe and Welch defined an uncountable cardinal $\kappa$ to be \emph{iterable} if for every subset $A$ of $\kappa$, there is a transitive model $M$ of $\ZFC^-$ of cardinality $\kappa$ with $A,\kappa\in M$ and a weakly amenable M-ultrafilter $U$ on $\kappa$ such that  $\langle M,\in,U\rangle$ is $\omega_1$-iterable. 
 It is easy to see that all iterable cardinals are stably measurable. Moreover, all Ramsey cardinals are iterable (see {\cite[Lemma 5.2]{MR2817562}}) and this shows that all measurable cardinals are stably measurable. 
 In the other direction, {\cite[Corollary 1.18]{zbMATH07415211}} shows that, if $\kappa$ is a stably measurable cardinal, then $a^\#$ exists for every set of ordinals $a$ in $\HH{\kappa}$. 
 Motivated by Corollary \ref{corollary:ConsequencesSingularNoMeasurable}, we proceed towards the following result:

 \begin{theorem}\label{theorem:StablyMeasurable}
  Let $\kappa$ be a stably measurable cardinal.
  \begin{enumerate}
   \item If $\mu<\kappa$ is an uncountable cardinal, then the singleton $\{\mu\}$ is $\Sigma_1(\HH{\mu})$-stationary in $\kappa$. 
  
   \item If $E$ is an unbounded subset of $\kappa$ that consists of cardinals, then $E$ is $\mathbf{\Sigma}_1$-stationary in $\kappa$. 

   \item If $S$ is a $\mathbf{\Sigma}_1$-stationary subset of $\kappa$ and $\map{r}{\kappa}{\kappa}$ is a regressive function that is definable by a $\Sigma_1$-formula with parameters in $\HH{\kappa}\cup\{\kappa\}$, then $r$ is constant on a $\mathbf{\Sigma}_1$-stationary subset of $S$. 
  \end{enumerate}
 \end{theorem}

 Using Lemma \ref{lemma:UnDefStationary} and Corollary \ref{corollary:UnbCardStat}, the first two statements of Theorem \ref{theorem:StablyMeasurable} directly follow from the next lemma:

\begin{lemma}\label{proposition:StablyMeasurableUndefinable}
 If $\kappa$ is a stably measurable cardinal, then $\kappa$ has the $\Sigma_1(\mu)$-undefinability property for every uncountable cardinal $\mu<\kappa$. 
\end{lemma}

\begin{proof}
 Assume, towards a contradiction, that there is a $\Sigma_1$-formula $\varphi(v_0,v_1,v_2)$, an uncountable cardinal $\mu<\kappa$, an ordinal $\alpha$ in the interval $[\mu,\kappa)$ and $z\in\HH{\mu}$ such that $\alpha$ is the unique ordinal $\xi$ with the property that $\varphi(\xi,\kappa,z)$ holds. 
 Pick a transitive set $M$ with $\HH{\kappa}\cup\{\kappa\}\subseteq M\prec_{\Sigma_1}\HH{\kappa^+}$, a transitive set $N$ with $M\cup{}^{{<}\kappa}N\subseteq N$ and a weakly amenable $N$-ultrafilter $F$ on $\kappa$ with the property that $\langle N,\in,F\rangle$ is iterable.  
 Now, pick an elementary submodel $\langle X,\in,F_0\rangle$ of $\langle N,\in,F\rangle$ of cardinality less than $\mu$ with $\tc{\{z\}}\cup\{\kappa,\alpha\}\subseteq X$, and let $\map{\pi}{X}{N_0}$ denote the corresponding transitive collapse. Then $\pi(z)=z$ and $\pi(\alpha)<\pi(\kappa)<\mu\leq\alpha$. Moreover, if we set $U=\pi[F_0]$, then $U$ is a weakly amenable $N_0$-ultrafilter on $\pi(\kappa)$ and {\cite[Theorem 19.15]{MR1994835}} implies that $\langle N_0,\in,U\rangle$ is iterable.  
 This yields a transitive set $N_1$ and an elementary embedding $\map{j}{N_0}{N_1}$ with $j(\pi(\kappa))=\kappa$ and $j\restriction\HH{\pi(\kappa)}^{N_0}=\id_{\HH{\pi(\kappa)}^{N_0}}$. 
 Since our setup ensures that $\varphi(\alpha,\kappa,z)$ holds in $N$, we now know that $\varphi(\pi(\alpha),\kappa,z)$ holds in $N_1$. By $\Sigma_1$-upwards absoluteness, this shows that $\varphi(\pi(\alpha),\kappa,z)$ holds in $\VV$, contradicting the uniqueness of $\alpha$. 
\end{proof}

 We now work towards a proof of the third part of Theorem \ref{theorem:StablyMeasurable}. The starting point for this is a result of Welch (see {\cite[Theorem 1.26]{zbMATH07415211}}) proving that stably measurable cardinal have the \emph{$\mathbf{\Sigma}_1$-club property} introduced in \cite{Sigma1Partitions}, i.e., if $\kappa$ is a stably measurable cardinal and $E$ is a subset of $\kappa$ with the property that the set $\{E\}$ is definable by a $\Sigma_1$-formula with parameters in $\HH{\kappa}\cup\{\kappa\}$, then either $E$ or $\kappa\setminus E$ contains a closed unbounded subset of $\kappa$. 
 This result directly implies that if $\kappa$ is a stably measurable cardinal and $\map{r}{\kappa}{\kappa}$ is a regressive function that is definable by a $\Sigma_1$-formula with parameters in $\HH{\kappa}\cup\{\kappa\}$, then for every ordinal $\alpha<\kappa$, there is a closed unbounded subset $C_\alpha$ of $\kappa$ with the property that $C_\alpha$ is either contained in $r^{{-}1}\{\alpha\}$ or disjoint from $r^{{-}1}\{\alpha\}$. By forming the diagonal intersection $\triangle_{\alpha<\kappa} C_\alpha$, it is now easy to see that there is a unique ordinal $\alpha_*<\kappa$ with the property that $C_{\alpha_*}\subseteq r^{{-}1}\{\alpha_*\}$ and this directly implies that $\alpha_*$ is the unique ordinal smaller than $\kappa$ with the property that $r^{{-}1}\{\alpha_*\}$ contains a closed unbounded subset of $\kappa$. The following lemma further strengthens this conclusion:

\begin{lemma}\label{lemma:DefClubStablyMeas}
 Let $\kappa$ be a stably measurable cardinal and let $E$ be a subset of $\kappa$  that contains a closed unbounded subset of $\kappa$ and has the property that the set $\{E\}$ is definable by a $\Sigma_1$-formula with parameters in $\HH{\kappa}\cup\{\kappa\}$. Then there exists a closed unbounded subset $C$ of $\kappa$ with $C\subseteq E$ and the property that the set $\{C\}$ is definable by a $\Sigma_1$-formula with parameters in $\HH{\kappa}\cup\{\kappa\}$. 
\end{lemma}

\begin{proof}
 Fix a $\Sigma_1$-formula $\varphi(v_0,v_1,v_2)$ and an element $z$ of $\HH{\kappa}$ with the property that the set $\{E\}$ is definable by the formula $\varphi(v_0,v_1,v_2)$ and the parameters $\kappa$ and $z$. By our assumptions on $\kappa$, there exists   a transitive set $M$ with $\HH{\kappa}\cup\{\kappa\}\subseteq M\prec_{\Sigma_1}\HH{\kappa^+}$, a transitive set $N$ with $M\cup{}^{{<}\kappa}N\subseteq N$ and a normal, weakly amenable $N$-ultrafilter $F$ on $\kappa$ with the property that $\langle N,\in,F\rangle$ is iterable.  This setup ensures that $E\in M$ and $\varphi(E,\kappa,z)$ holds in $N$. Moreover, the fact that $E$ contains a closed unbounded subset of $\kappa$ implies that $M$ contains such a subset of $E$ and, since $F$ is a normal $N$-ultrafilter, we can conclude that $E$ is an element of $F$. 
 Pick an elementary submodel $\langle X,\in,F_0\rangle$ of $\langle N,\in,F\rangle$ of cardinality less than $\kappa$ with $\tc{\{z\}}\cup\{\kappa\}\subseteq X$, and let $\map{\pi}{X}{N_0}$ denote the corresponding transitive collapse. Then $\pi(z)=z$ and, if we set $U=\pi[F_0]$, then $U$ is a weakly amenable $N_0$-ultrafilter on $\pi(\kappa)$ and {\cite[Theorem 19.15]{MR1994835}} ensures that $\langle N_0,\in,U\rangle$ is iterable. 
 Let $$\langle\seq{N_\alpha}{\alpha\in Ord},\seq{\map{j_{\alpha,\beta}}{M_\alpha}{M_\beta}}{\alpha\leq\beta\in Ord}\rangle$$ denote the linear iteration of $\langle N_0,\in,U\rangle$. 
 Then $(j_{0,\kappa}\circ\pi)(\kappa)=\kappa$, $(j_{0,\kappa}\circ\pi)(z)=z$, $\Sigma_1$-upwards absoluteness implies that $\varphi((j_{0,\kappa}\circ\pi)(E),\kappa,z)$ holds in $V$ and hence  $(j_{0,\kappa}\circ\pi)(E)=E$. Moreover, the fact that $E\in F$ ensures that   $\pi(E)\in U$ and therefore $(j_{0,\alpha}\circ\pi)(\kappa)\in(j_{0,\alpha+1}\circ\pi)(E)\subseteq E$ for all $\alpha<\kappa$. This shows that the closed unbounded subset $C=\Set{(j_{0,\alpha}\circ\pi)(\kappa)}{\alpha<\kappa}$ of $\kappa$ is a subset of $E$. Finally, the set $\{C\}$ is definable by a $\Sigma_1$-formula with parameters $\kappa$, $N_0$ and $U$. 
\end{proof}

\begin{proof}[Proof of Theorem \ref{theorem:StablyMeasurable}]
The last two lemmata prove the first two assertions of the theorem. 
To prove the third assertion, let $\kappa$ be a stably measurable cardinal, let $S$ be $\mathbf{\Sigma}_1$-stationary in $\kappa$, and let $\map{r}{\kappa}{\kappa}$ be a regressive function that is definable by a $\Sigma_1$-formula with parameters in $\HH{\kappa}\cup\{\kappa\}$. Our earlier observations now show that there is a unique ordinal $\alpha<\kappa$ with the property that the set $r^{{-}1}\{\alpha\}$ contains a closed unbounded subset of $\kappa$. Since the set $\{r^{{-}1}\{\alpha\}\}$ is definable by a $\Sigma_1$-formula with parameters in $\HH{\kappa}\cup\{\kappa\}$, Lemma \ref{lemma:DefClubStablyMeas} yields a closed unbounded subset $C$ of $\kappa$ with $C\subseteq r^{{-}1}\{\alpha\}$ and the property that the set $\{C\}$ is definable by a $\Sigma_1$-formula with parameters in $\HH{\kappa}\cup\{\kappa\}$. 
 Then $r$ is constant on $C\cap S$ and, since $\kappa$ is an uncountable regular cardinal, we know that this set is $\mathbf{\Sigma}_1$-stationary in $\kappa$.  
\end{proof}

We can now show that, in the Dodd--Jensen core model $\KK^{DJ}$, stably measurable cardinals are characterized by the $\mathbf{\Sigma}_1$-stationarity of unbounded sets of cardinals:

 \begin{theorem}\label{theorem:InaccNonDefInKDJ}
 If $\VV=\KK^{DJ}$, then the following statements are equivalent for every  cardinal $\kappa>\omega_1$: 
 \begin{enumerate}
  \item The cardinal $\kappa$ is stably measurable. 

  \item The set $\{\HH{\kappa}\}$ is not definable by a $\Sigma_1$-formula with parameters in the set $\HH{\kappa}\cup\{\kappa\}$. 

  \item The cardinal $\kappa$ is a limit cardinal and every unbounded subset of $\kappa$ that consists of cardinals  is $\mathbf{\Sigma}_1$-stationary in $\kappa$. 
 \end{enumerate}
\end{theorem} 

\begin{proof} 
 First, assume that $\kappa$ is not stably measurable. 
 Then we can apply {\cite[Theorem 2.6]{zbMATH07415211}} to  find a bistationary subset $E$ of $\kappa$ with the property that the set $\{E\}$ is definable by a $\Sigma_1$-formula with parameters in $\HH{\kappa}\cup\{\kappa\}$.  Lemma \ref{lemma:DefineFromBistationary} then shows that the set $\{\HH{\kappa}\}$ is definable by a $\Sigma_1$-formula with parameters in $\HH{\kappa}\cup\{\kappa\}$.

 Now, assume that $\kappa$ is a limit cardinal and the set $\{\HH{\kappa}\}$ is definable by a $\Sigma_1$-formula with parameters in $\HH{\kappa}\cup\{\kappa\}$. 
 Let $C$ denote the set of all limit cardinals smaller than $\kappa$. Our assumption then implies that the set $\{C\}$ is definable by a $\Sigma_1$-formula with parameters in $\HH{\kappa}\cup\{\kappa\}$. 
 If $C$ is unbounded in $\kappa$, then the set of all successor cardinals smaller $\kappa$ is an unbounded subset of $\kappa$ that consists of cardinals and is not $\mathbf{\Sigma}_1$-stationary in $\kappa$. In the other case, if $C$ is bounded in $\kappa$, then $\cof{\kappa}=\omega$, Lemma \ref{lemma:LeastCofinal} yields a cofinal function $\map{c}{\omega}{\kappa}$ that is definable by a $\Sigma_1$-formula with parameter $\kappa$ and the set $D=\Set{c(n)+1}{n<\omega}$ is closed unbounded in $\kappa$ that contains no cardinals and has the  property that the set $\{D\}$ is definable by a $\Sigma_1$-formula with parameter $\kappa$. 

 Together with Theorem  \ref{theorem:StablyMeasurable}, these computations establish all equivalences claimed in the theorem. 
\end{proof}

Next, we show that stable measurability provides the exact consistency strength for the existence of limit cardinals with given  undefinability property:

 \begin{theorem}\label{theorem:ConsStrengthInacc}
 The following statements are equiconsistent over $\ZFC$: 
 \begin{enumerate}
  \item There exists a stably measurable cardinal. 

  \item There exists a limit cardinal $\kappa$ with the property that every unbounded subset of $\kappa$ that consists of cardinals is $\mathbf{\Sigma}_1$-stationary in $\kappa$. 
 \end{enumerate} 
\end{theorem}

\begin{proof} 
 Assume, towards a contradiction, that there is no inner model with a stably measurable cardinal and $\kappa$ is a limit cardinal with the property that every unbounded subset of $\kappa$ that  consists of cardinals is $\mathbf{\Sigma}_1$-stationary in $\kappa$.  Then Corollary \ref{corollary:ConsequencesSingularNoMeasurable} shows that $\kappa$ is regular.   
 Since $\kappa$ is not stably measurable in $\KK^{DJ}$, Theorem \ref{theorem:InaccNonDefInKDJ} now shows that, in $\KK^{DJ}$, the set $\{\HH{\kappa}\}$ is definable by a $\Sigma_1$-formula with parameters in $\HH{\kappa}\cup\{\kappa\}$. 
 Let $C$ denote the set of all ordinals less than $\kappa$ that are limit cardinals in $\KK^{DJ}$. Then $C$ is a closed unbounded subset of $\kappa$ and, since the class $\KK^{DJ}$ is definable by a $\Sigma_1$-formula with parameter $\kappa$, the set $\{C\}$ is definable by a $\Sigma_1$-formula with parameters in $\HH{\kappa}\cup\{\kappa\}$. Now, let $E$ denote the set of all successor cardinals of singular cardinals less than $\kappa$. The results of \cite{MR661475} then show that the elements of $E$ are successor cardinals of singular cardinals in $\KK^{DJ}$ and hence $C\cap E=\emptyset$. This shows that $E$ is an unbounded subset of $\kappa$ that consists of cardinals and is not $\mathbf{\Sigma}_1$-stationary.  
 In combination with Theorem \ref{theorem:StablyMeasurable}, these arguments prove the desired equivalence.  
\end{proof}

In Section \ref{Section:ConsistencyResults} we show that the assertion in item (2) of the last theorem can consistently hold at $\kappa = \aleph_\omega$, starting from the consistency assumption of a measurable cardinal. 


\subsection{Limits of measurable cardinals}

 We now continue by showing that many of the above results  can be extended to larger classes of definable sets if the given cardinal is a (not necessarily regular) limit of measurable cardinals:

  \begin{theorem}\label{theorem:LimitMeasurables}
  Let $\kappa$ be a cardinal that is a limit of measurable cardinals. 
  \begin{enumerate}
   \item Every unbounded subset  $S$ of $\kappa$ consisting of cardinals is $\mathbf{\Sigma}_1(Ord)$-stationary. 

   \item If $S$ is a $\mathbf{\Sigma}_1(Ord)$-stationary subset of $\kappa$ and $\map{r}{\kappa}{\kappa}$ is a regressive function that is definable by a $\Sigma_1$-formula with parameters in $\HH{\kappa}\cup Ord$, then $r$ is constant on a $\mathbf{\Sigma}_1(Ord)$-stationary subset of $S$. 
  \end{enumerate}
 \end{theorem}

 The starting point of the proof of this theorem is the following technical lemma that generalizes {\cite[Lemma 1.1.25]{La04}} to embeddings given by linear iterations.

\begin{lemma}\label{lemma:CommutingEmbeddings}
 Suppose that  $\kappa_0<\kappa_1$ are measurable cardinal and $\alpha_0 < \kappa_1$ is an ordinal. Given $i<2$, let $U_i$ be a normal ultrafilter on $\kappa_i$ and let $$\langle\seq{M^i_\alpha}{\alpha\in Ord},\seq{\map{j^i_{\alpha,\beta}}{M^i_\alpha}{M^i_\beta}}{\alpha\leq\beta\in Ord}\rangle$$ denote the linear iteration of $\langle\VV,\in,U_i\rangle$. In addition, let $$\langle\seq{M^*_\alpha}{\alpha\in Ord},\seq{\map{j^*_{\alpha,\beta}}{M^*_\alpha}{M^*_\beta}}{\alpha\leq\beta\in Ord}\rangle$$ denote the linear iteration of $\langle M^0_{\alpha_0},\in,j^0_{0,\alpha_0}(U_1)\rangle$. Then $j^*_{0,\alpha}\restriction Ord=j^1_{0,\alpha}\restriction Ord$ for all $\alpha\in Ord$. 
\end{lemma}

\begin{proof}
 Fix an ordinal $\alpha_1<\kappa$ with the property that $j^*_{0,\alpha}\restriction Ord=j^1_{0,\alpha}\restriction Ord$ holds for all $\alpha<\alpha_1$. Set $M^0=M^0_{\alpha_0}$, $M^1=M^1_{\alpha_1}$, $M^*=M^*_{\alpha_1}$, $j^0=j^0_{0,\alpha_0}$, $j^1=j^1_{0,\alpha_1}$,  $j^*=j^*_{0,\alpha_1}$ and  
 $U_*=j^0(U_1)$. Then $j^0(\kappa_1)=\kappa_1$. 
 In addition, if $i<2$, then we define $\kappa^i_\alpha=j^i_{0,\alpha}(\kappa_i)$ for all $\alpha\in Ord$ and $C^i=\Set{\kappa^i_\alpha}{\alpha<\alpha_i}$. 
 An application of {\cite[Lemma 19.6]{MR1994835}} then shows that every element of $M^*$ is of the form $j^0(f)(c)$ with $n<\omega$, $\map{f}{[\kappa_0]^n}{V}$ and $c\in[C^0]^n$. 
 In addition, given $i<2$ and $n<\omega$, we define $$U_i^n ~ = ~ \Set{X\subseteq [\kappa_i]^n}{\exists E\in U_i ~ [E]^n\subseteq X}.$$ The results of {\cite[Chapter 19]{MR1994835}} then show that $U_i^n$ is a ${<}\kappa_i$-complete ultrafilter on $[\kappa_i]^n$. Set $U_*^n=j^0(U_1^n)$ for all $n<\omega$. Our assumption now  implies that $j^*_{0,\alpha}(\kappa_1)=\kappa^1_\alpha$ holds for all $\alpha<\alpha_1$. In particular, we know that every ordinal is of the form $j^*(g)(d)$ with $n<\omega$, $\map{g}{[\kappa_1]^n}{Ord}$ in $M_0$ and $d\in [C^1]^n$.

 \begin{claim*}
  If $n<\omega$, $B\in U_1^n$ and $B_*\in U_*^n$, then $j^0[B]\cap B_*\neq\emptyset$. 
 \end{claim*}

 \begin{proof}[Proof of the Claim]
  Pick $n<\omega$, $\map{f}{[\kappa_0]^n}{V}$ and $c\in [C^0]^n$ with $B_*=j^0(f)(c)$. 
  By {\cite[Lemma 19.9]{MR1994835}}, we then have $$D ~ = ~ \Set{a\in[\kappa_0]^n}{f(a)\in U_1^n} ~ \in ~ U_0^n.$$  
  Since $U_0^n$ has cardinality less than $\kappa_1$, the ${<}\kappa_1$-completeness of $U_1^n$ implies that $E=\bigcap\Set{f(a)}{a\in D}$ is an element of $U_1^n$. 
  Pick $b\in B\cap E$. 
  Then $$D ~ \subseteq ~ \Set{a\in[\kappa_0]^n}{b\in f(a)} ~ \in ~ U_0^n$$ and {\cite[Lemma 19.9]{MR1994835}} shows that $j^0(b)\in j^0[B]\cap j^0(f)(c)=j^0[B]\cap B_*\neq\emptyset$. 
 \end{proof}

 \begin{claim*}
  If $n<\omega$ and $\map{g_0,g_1}{[\kappa_1]^n}{Ord}$ are functions in $M^0$ with $$\Set{b\in[\kappa_1]^n}{g_0(b)=g_1(b)}\in U_*^n,$$ then $$\Set{b\in[\kappa_1]^n}{g_0(j^0(b))=g_1(j^0(b))}\in U_1^n.$$
 \end{claim*}

 \begin{proof}[Proof of the Claim]
  Assume, towards a contradiction, that the above conclusion fails. Then we know that the set $B=\Set{b\in[\kappa_1]^n}{g_0(j^0(b))\neq g_1(j^0(b))}$ is an element of $U_1^n$. In this situation, we can use our first claim to find $b\in B$ with $g_0(j^0(b))=g_1(j^0(b))$, a contradiction. 
 \end{proof}

 The same proof also yields the following implication: 

  \begin{claim*}
  If $n<\omega$ and $\map{g_0,g_1}{[\kappa_1]^n}{Ord}$ are functions in $M^0$ such that the set  $\Set{b\in[\kappa_1]^n}{g_0(b)<g_1(b)}$ is an element of $U_*^n$,  then the set $$\Set{b\in[\kappa_1]^n}{g_0(j^0(b))<g_1(j^0(b))}$$ is an element of $U_1^n$. \qed
 \end{claim*}

 \begin{claim*}
  For every function $\map{f}{[\kappa_1]^n}{Ord}$ in $V$, there is a function $\map{g}{[\kappa_1]^n}{Ord}$ in $M^0$ with $$\Set{b\in[\kappa_1]^n}{f(b)=g(j^0(b))}\in U_1^n.$$
 \end{claim*}

 \begin{proof}[Proof of the Claim]
  Given $b\in[\kappa_1]^n$, pick $m_b<\omega$, $\map{f_b}{[\kappa_0]^{m_b}}{Ord}$ in $V$ and $c_b\in[C^0]^{m_b}$ with $f(b)=j^0(f_b)(c_b)$. Then there is $m<\omega$ and $c\in[C^0]^m$ with $$B ~ = ~ \Set{b\in[\kappa_1]^n}{m_b=m, ~ c_b=c}  ~ \in ~ U_1^n.$$
  Define $$\Map{G}{[\kappa_0]^m\times B}{Ord}{(a,b)}{f_b(a)}$$ and $$\Map{g}{[\kappa_1]^n}{Ord}{b}{j^0(G)(c,b)}.$$ Then $g$ is an element of $M^0$ and $$g(j^0(b)) ~ = ~ j^0(G)(c,j^0(b)) ~ = ~ j^0(f_b)(c) ~ = ~ f(b)$$ holds for all $b\in B\in U_1^n$. 
 \end{proof}

 \begin{claim*}
  If $\map{g}{[\kappa_1]^1}{Ord}$ is an element of $M^0$, then $$j^1(g\circ j^0)(\{\kappa_1\}) ~ = ~ j^1(g)(\{\kappa_1\}).$$
 \end{claim*}

 \begin{proof}[Proof of the Claim]
  Since standard arguments show that  $\Set{\{\gamma\}}{\gamma<\kappa_1, j^0(\gamma)=\gamma}$ is an element of $U^1_1$, we know that the set $\Set{\{\gamma\}}{\gamma<\kappa_1, (g\circ j^0)(\gamma)=g(\gamma)}$ is also contained in $U_1^1$ and hence an application of {\cite[Lemma 19.9]{MR1994835}} shows that we have $j^1(g\circ j^0)(\{\kappa_1\})  =  j^1(g)(\{\kappa_1\})$. 
 \end{proof}

 By our second claim and {\cite[Lemma 19.9]{MR1994835}}, there is a well-defined class function $\map{F}{Ord}{Ord}$ with the property that $$F(j^*(g)(c)) ~ = ~ j^1(g\circ j^0)(c)$$ holds for all $n<\omega$, $\map{g}{[\kappa_1]^n}{Ord}$ in $M^0$ and $c\in [C^1]^n$. Moreover, our third claim shows that $F$ is order-preserving. Finally, since every ordinal is of the form $j^1(f)(c)$ for some $n<\omega$, $\map{f}{[\kappa_1]^n}{Ord}$ in $V$ and $c\in[C^1]^n$, our fourth claim shows that $F$ is also surjective. In combination, this shows that $F$ is the identity on $Ord$. 
 Now,    fix an ordinal $\gamma$ and let $g_\gamma$ denote the constant function on $[\kappa_1]^1$ with value $\gamma$. Then we have 
 \begin{equation*}
  \begin{split}
   j^*(\gamma) ~ & = ~ j^*(g_\gamma)(\{\kappa_1\}) ~ = ~ F(j^*(g_\gamma)(\{\kappa_1\}))  \\ &  = ~ j^1(g_\gamma\circ j^0)(\{\kappa_1\})  ~ = ~ j^1(g_\gamma)(\{\kappa_1\}) ~ = ~ j^1(\gamma)
  \end{split}
 \end{equation*}
 and this proves the statement of the lemma.  
\end{proof}

Using the above results, we can now drive a variation of {\cite[Lemma 1.1.27]{La04}} for iterations. The following lemma   will be the main tool used in the proofs of this section.

\begin{lemma}\label{lemma:LimitMeasurableFixedPoints}
 Let $\lambda \geq \aleph_0$ be a regular cardinal, let $\vec{\kappa}=\seq{\kappa_\xi}{\xi<\lambda}$ be a strictly increasing sequence of measurable cardinals with supremum $\kappa$ and let $\vec{U}=\seq{U_\xi}{\xi<\lambda}$ be a sequence with the property that $U_\xi$ is a normal ultrafilter on $\kappa_\xi$ for all $\xi<\lambda$. Assume that either $\kappa=\lambda$ or $\lambda<\kappa_0$. 
 Given $\xi<\lambda$, let $$\langle\seq{M^\xi_\alpha}{\alpha\in Ord},\seq{\map{j^\xi_{\alpha,\beta}}{M^\xi_\alpha}{M^\xi_\beta}}{\alpha\leq\beta\in Ord}\rangle$$ denote the linear iteration of $\langle\VV,\in,U_\xi\rangle$. 
 Then for all ordinals $\gamma$ and eventually all $\xi<\lambda$, we have $j^\xi_{0,\alpha}(\gamma)=\gamma$  for all   $\alpha<\kappa$. 
\end{lemma}

\begin{proof}
 Assume, towards a contradiction, that there is an ordinal $\gamma$ such that for unboundedly many $\xi<\lambda$, there is an $\alpha<\kappa$ with $j^\xi_{0,\alpha}(\gamma)>\gamma$.  
 Let $\gamma$ be minimal with this property. Pick $\zeta<\lambda$ and  $\alpha_0<\kappa$ with $j^\zeta_{0,\alpha_0}(\gamma)>\gamma$. Set $j=j^\zeta_{0,\alpha_0}$ and $M=M^\zeta_{\alpha_0}$. Then $j(\lambda)=\lambda$ and $j(\kappa)=\kappa$. Moreover, if we set $j(\vec{\kappa})=\seq{\kappa^\prime_\xi}{\xi<\lambda}$ and $j(\vec{U})=\seq{U^\prime_\xi}{\xi<\lambda}$, then the fact that $j(\gamma)>\gamma$ yields an ordinal $\rho<\lambda$ with the property that $i^\xi_{0,\alpha}(\gamma)=\gamma$ holds for all $\rho\leq\xi<\lambda$ and all $\alpha<\kappa$, where $$\langle\seq{N^\xi_\alpha}{\alpha\in Ord},\seq{\map{i^\xi_{\alpha,\beta}}{N^\xi_\alpha}{N^\xi_\beta}}{\alpha\leq\beta\in Ord}\rangle$$ denotes the linear iteration of $\langle M,\in,U^\prime_\xi\rangle$. 
 By our assumption, we can now pick $\rho<\xi<\lambda$ and $\alpha_1<\kappa$ such that $\kappa_\xi>\alpha_0$ and $j^\xi_{0,\alpha_1}(\gamma)>\gamma$. In this situation, Lemma \ref{lemma:CommutingEmbeddings} implies that $i^{j(\xi)}_{0,\alpha_1}(\gamma)=j^\xi_{0,\alpha_1}(\gamma)>\gamma$ and this yields a contradiction, because $\rho\leq\xi\leq j(\xi)<\lambda$. 
\end{proof}

 With the help of Lemma \ref{lemma:LimitMeasurableFixedPoints}, we can now work towards a proof of the first part of Theorem \ref{theorem:LimitMeasurables}.

\begin{lemma}\label{lemma:LimitMeasurableUndefProp}
  In the setting of  Lemma \ref{lemma:LimitMeasurableFixedPoints}, for every ordinal $\gamma\geq\kappa$, there exists an ordinal $\lambda_0<\lambda$ with the property that the cardinal $\kappa$ has the $\Sigma_1(\kappa_\xi,\gamma)$-undefinability property for all $\lambda_0\leq\xi<\lambda$.   
\end{lemma}

\begin{proof}
  Lemma \ref{lemma:LimitMeasurableFixedPoints} yields $\lambda_0<\lambda$ with the property that $j^\xi_{0,\alpha}(\kappa)=\kappa$ and $j^\xi_{0,\alpha}(\gamma)=\gamma$ holds for all $\lambda_0\leq\xi<\lambda$ and $\alpha<\kappa$. 
  Assume, towards a contradiction, that there is a $\Sigma_1$-formula $\varphi(v_0,\ldots,v_3)$, and ordinal $\lambda_0\leq\xi<\lambda$,  an element $z$ of $\HH{\kappa_\xi}$ and an ordinal $\alpha$ in the interval $[\kappa_\xi,\kappa)$ such that the set $\{\alpha\}$ is definable by the formula $\varphi(v_0,\ldots,v_3)$ and the parameters $\kappa$, $\gamma$ and $z$. Pick an ordinal $\beta<\kappa$ with $j^\xi_{0,\beta}(\kappa_\xi)>\alpha$. Then $j^\xi_{0,\beta}(\alpha)>\alpha$ and elementarity implies that $\varphi(j^\xi_{0,\beta}(\alpha),\kappa,\gamma,z)$ holds in $M^\xi_\beta$. But, then $\Sigma_1$-upwards absoluteness implies that $\varphi(j^\xi_{0,\beta}(\alpha),\kappa,\gamma,z)$ holds in $V$, contradicting our assumptions. 
\end{proof}

\begin{lemma}\label{lemma:LimitMeasurableStationarySets}
 Let $\kappa$ be a cardinal that is a limit of measurable cardinals. Then every unbounded subset of $\kappa$ that consists of cardinals is $\mathbf{\Sigma}_1(Ord)$-stationary in $\kappa$. 
\end{lemma}

\begin{proof}
 Work in the setting of  Lemma \ref{lemma:LimitMeasurableFixedPoints}. Let $E$ be an unbounded subset of $\kappa$ that consists of cardinals and  fix a $\Sigma_1$-formula $\varphi(v_0,v_1,v_2)$, an ordinal $\gamma$ and $z\in\HH{\kappa}$ with the property that there exists a closed unbounded subset $C$ of $\kappa$ that is the unique set $x$ with the property that $\varphi(x,\gamma,z)$ holds. 
 By combining Lemma \ref{lemma:LimitMeasurableFixedPoints} and Lemma \ref{lemma:LimitMeasurableUndefProp} with Corollary \ref{corollary:UnbCardStat}, we can find $\xi<\lambda$ such that $\kappa_\xi$ is an element of $C$, $z\in\HH{\kappa_\xi}$ and for all $\alpha<\kappa$, we have $j^\xi_{0,\alpha}(\kappa)=\kappa$ and $j^\xi_{0,\alpha}(\gamma)=\gamma$. 
 In this situation, we can find a cardinal $\kappa_\xi<\mu\in E$ with $j^\xi_{0,\mu}(\kappa_\xi)=\mu$. Then $\varphi(j^\xi_{0,\mu}(C),\gamma,z)$ holds in $M^\xi_\mu$ and, since $\Sigma_1$-upwards absoluteness implies that this statement also holds in $\VV$, it follows that $j^\xi_{0,\mu}(C)=C$. 
 Moreover, since $\kappa_\xi\in C$, we can now use elementarity to conclude that $\mu=j^\xi_{0,\mu}(\kappa_\xi)\in C\cap E\neq\emptyset$. 
\end{proof}

 The next result generalizes {\cite[Theorem 1.26]{zbMATH07415211}} and Lemma \ref{lemma:DefClubStablyMeas} to (possibly singular) limits of singular cardinals:

\begin{lemma}\label{lemma:Sigma1ClubProperty}
  Let $\kappa$ be a cardinal that is a limit of measurable cardinals and let $E$ be a subset of $\kappa$ with the property that the set $\{E\}$ is definable by a $\Sigma_1$-formula with parameters in $Ord\cup\HH{\kappa}$. Then there exists a closed unbounded subset $C$ of $\kappa$ of order type $\kappa$ with the property that the set $\{C\}$ is definable by a $\Sigma_1$-formula with parameters in $\HH{\kappa}\cup\{\kappa\}$ and either $C\subseteq E$ or $C\cap E=\emptyset$ holds.  
\end{lemma}

\begin{proof}
 We may assume that there is a $\Sigma_1$-formula $\varphi(v_0,v_1,v_2)$, an ordinal $\gamma$ and $z\in\HH{\kappa}$ with the property that $E$ is the unique set $x$ satisfying $\varphi(x,\gamma,z)$. By Lemma \ref{lemma:LimitMeasurableFixedPoints}, there exists a measurable cardinal $\kappa_0<\kappa$ with $z\in\HH{\kappa_0}$ and a normal ultrafilter $U$ on $\kappa_0$ satisfying $j_{0,\alpha}(\kappa)=\kappa$ and $j_{0,\alpha}(\gamma)=\gamma$ for all $\alpha<\kappa$, where $$\langle\seq{M_\alpha}{\alpha\in Ord},\seq{\map{j_{\alpha,\beta}}{M_\alpha}{M_\beta}}{\alpha\leq\beta\in Ord}\rangle$$ is the linear iteration of $\langle V,\in,U\rangle$. 
 Given an ordinal $\alpha<\kappa$, elementarity ensures that $\varphi(j_{0,\alpha}(E),\gamma,z)$ holds in $M_\alpha$. In this situation, $\Sigma_1$-upwards absoluteness implies that $\varphi(j_{0,\alpha}(E),\gamma,z)$ holds in $V$ for all $\alpha<\kappa$, and this allows us to conclude that $j_{0,\alpha}(E)=E$ holds for all $\alpha<\kappa$. 
 If $\kappa_0$ is an element of $E$, then these observations show that $\Set{j_{0,\alpha}(\kappa_0)}{\alpha<\kappa}$ is a closed unbounded subset of $E$ of order type $\kappa$. In the other case, if $\kappa_0\notin E$, then $\Set{j_{0,\alpha}(\kappa_0)}{\alpha<\kappa}$ is a closed unbounded subset of $E$ of order type $\kappa$ that is disjoint from $E$. 

 Now, pick  a transitive model $N$ of $\ZFC^-$ such that $N$ contains $\kappa$ and $U$ and the model $\langle N,\in,U\rangle$ is linearly iterable. Let $$\langle\seq{N_\alpha}{\alpha\in Ord},\seq{\map{i_{\alpha,\beta}}{N_\alpha}{N_\beta}}{\alpha\leq\beta\in Ord}\rangle$$ denote the linear iteration of $\langle N,\in,U\rangle$. In this situation, the fact that $\POT{\kappa_0}\subseteq N$ implies that $j_{0,\alpha}(\kappa_0)=i_{0,\alpha}(\kappa_0)$ holds for all $\alpha<\kappa$. Since the class of all sets $N$ with the above properties is definable by a $\Sigma_1$-formula with parameter $U$, we can conclude that the set $\{\Set{j_{0,\alpha}(\kappa_0)}{\alpha<\kappa}\}$ is definable by a $\Sigma_1$-formula with parameters in $\HH{\kappa}\cup\{\kappa\}$. 
\end{proof}

 The techniques used in the above proofs also allow us to prove the second part of Theorem \ref{theorem:LimitMeasurables}:

\begin{proof}[Proof of Theorem \ref{theorem:LimitMeasurables}]
 Let $\kappa$ be a limit of measurable cardinals. Then Lemma \ref{lemma:LimitMeasurableStationarySets} shows that every unbounded subset of $\kappa$ that consists of cardinals is $\mathbf{\Sigma}_1(Ord)$-stationary. 
 Now, let $S$ be a $\mathbf{\Sigma}_1(Ord)$-stationary subset of $\kappa$ and let $\map{r}{\kappa}{\kappa}$ be a regressive function that is definable by a $\Sigma_1$-formula $\varphi(v_0,\ldots,v_3)$, an ordinal $\gamma$ and an element $z$ of $\HH{\kappa}$. As in the proof of Lemma \ref{lemma:Sigma1ClubProperty}, we can now find a measurable cardinal $\kappa_0<\kappa$ with $z\in \HH{\kappa_0}$ and a normal ultrafilter $U$ on $\kappa_0$ with the property that $j_{0,\alpha}(\gamma)=\gamma$ holds for all $\alpha<\kappa$, where $$\langle\seq{M_\alpha}{\alpha\in Ord},\seq{\map{j_{\alpha,\beta}}{M_\alpha}{M_\beta}}{\alpha\leq\beta\in Ord}\rangle$$ denotes the linear iteration of $\langle V,\in,U\rangle$. Let $r(\kappa_0)=\xi<\kappa_0$. Then $\varphi(\kappa_0,\xi,\gamma,z)$ holds in $V$ and for all $\alpha<\kappa$, elementarity implies that $\varphi(j_{0,\alpha}(\kappa_0),\xi,\gamma,z)$ holds in $M_\alpha$. Given $\alpha<\kappa$,  $\Sigma_1$-upwards absoluteness now implies that $\varphi(j_{0,\alpha}(\kappa_0),\xi,\gamma,z)$ holds in $V$ and hence $r(j_{0,\alpha}(\kappa_0))=\xi$. This shows that the restriction of $r$ to the closed unbounded subset $C=\Set{j_{0,\alpha}(\kappa_0)}{\alpha<\kappa}$ of $\kappa$ is constant with value $\xi$. Moreover, the proof of Lemma \ref{lemma:Sigma1ClubProperty} shows that the set $\{C\}$ is definable by a $\Sigma_1$-formula with parameters $\kappa$ and $U$. 
 Since Corollary \ref{corollary:CriterionClosedIntersections} and Lemma \ref{lemma:LimitMeasurableStationarySets} show that the set $C\cap S$ is $\mathbf{\Sigma}_1(Ord)$-stationary in $\kappa$, these arguments show that $r$ is constant on a $\mathbf{\Sigma}_1(Ord)$-stationary subset of $S$.  
\end{proof}


\subsection{Partition properties}\label{subsection:PartitionProperty}

 Remember that, given uncountable cardinals $\mu<\kappa$, the cardinal $\kappa$ is \emph{$\mu$-Rowbottom} if the square brackets  partition relation $\kappa\longrightarrow[\kappa]^{{<}\omega}_{\lambda,{<}\mu}$ holds true for all $\lambda<\kappa$, i.e., for every $\lambda<\kappa$ and every function $\map{c}{[\kappa]^{{<}\omega}}{\lambda}$, there exists $H\in[\kappa]^\kappa$ with $\betrag{c[[H]^{{<}\omega}]}<\mu$. Moreover, $\omega_1$-Rowbottom cardinals are called \emph{Rowbottom cardinals}.  
 The following lemma connects this partition property to the $\mathbf{\Sigma}_1$-undefinability property:

\begin{lemma}\label{lemma:Rowbottem}
 Let $\kappa$ be a $\mu$-Rowbottom cardinal. 
 If either $\kappa=\omega_\omega$ or $\rho^{{<}\mu}<\kappa$ holds for all $\rho<\kappa$, then $\kappa$ has the $\Sigma_1(\mu)$-undefinability property. 
\end{lemma}

\begin{proof}
  Assume, towards a contradiction, that there is a $\Sigma_1$-formula $\varphi(v_0,v_1,v_2)$, an element $z$ of $\HH{\mu}$ and an ordinal $\alpha$ in $[\mu,\kappa)$ that is the unique set $x$ such that $\varphi(x,\kappa,z)$ holds. 

  \begin{claim*}
   There is an elementary submodel $X$ of $\HH{\kappa^+}$ with $\tc{\{z\}}\cup\{\kappa,\alpha\}\subseteq X$, $\betrag{X\cap\alpha}<\betrag{\alpha}$ and $\betrag{X\cap\kappa}=\kappa$. 
  \end{claim*}

  \begin{proof}[Proof of the Claim]
   Define $\calL$ to be the countable first-order language that extends $\calL_\in$ by a unary predicate symbol $\dot{P}$, a unary function symbol $\dot{s}$, a constant symbol $\dot{\kappa}$, a constant symbol $\dot{\alpha}$ and a constant symbol $\dot{c}_y$ for every $y\in\tc{\{z\}}$. 
  Pick an elementary submodel $M$ of $\HH{\kappa^+}$ of cardinality $\kappa$ with $\tc{\{z\}}\cup(\kappa+1)\subseteq M$ and a surjection $\map{s}{\kappa}{M}$. 
  Now, let $A$ denote an $\calL$-expansion of $\langle M,\in\rangle$ with $\dot{P}^A=\alpha$, $\dot{s}^A\restriction\kappa=s$, $\dot{\kappa}^A=\kappa$, $\dot{\alpha}^A=\alpha$ and $\dot{c}_y^A=y$ for all $y\in\tc{\{z\}}$. In addition, we define  $\rho=\betrag{\alpha}\geq\mu$. 

  First, assume that $\rho^{{<}\mu}<\kappa$ holds. Then {\cite[Theorem 8.5]{MR1994835}} yields an elementary substructure $X$ of $A$ of cardinality $\kappa$ with  $\betrag{X\cap \alpha}<\mu\leq\rho$. This setup then ensures that $\tc{\{z\}}\cup\{\kappa,\alpha\}\subseteq X$ and, since $\ran{s\restriction X}=X$, we also know that $\betrag{X\cap\kappa}=\kappa$. 

  In the other case, assume that $\kappa=\omega_\omega$. Then $\rho$ is regular and $\kappa$ is also $\rho$-Rowbottom. Another application of {\cite[Theorem 8.5]{MR1994835}} now produces an elementary substructure $X$ of $A$ of cardinality $\kappa$ with  $\betrag{X\cap \alpha}<\rho$. As above, we can conclude that $\tc{\{z\}}\cup\{\kappa,\alpha\}\subseteq X$ and $\betrag{X\cap\kappa}=\kappa$.  
  \end{proof}

  Let $\map{\pi}{X}{N}$ denote the corresponding transitive collapse. Then our setup ensures that $\pi(\kappa)=\kappa$, $\pi(z)=z$ and $\pi(\alpha)<\alpha$. Moreover, since $\Sigma_1$-absoluteness causes $\varphi(\alpha,\kappa,z)$ to holds in $\HH{\kappa^+}$ and $X$, we know that $\varphi(\pi(\alpha),\kappa,z)$ holds in $N$. But then $\Sigma_1$-upwards absoluteness implies that this statement holds in $V$, contradicting our assumptions. 
\end{proof}

 Recall that a cardinal $\kappa$ is \emph{J\'{o}nsson} if for every function $\map{f}{[\kappa]^{{<}\omega}}{\kappa}$ there is a proper subset $H$ of $\kappa$ of cardinality $\kappa$ with $f[[H]^{{<}\omega}]\subseteq H$. 
 Motivated by the notoriously open question whether the first limit cardinal $\omega_\omega$ can be J\'{o}nsson, we show that this assumption causes analogs of central results from the previous two sections to hold at $\omega_\omega$.

 \begin{theorem}\label{theorem:JonssonStationary}
  If  $\omega_\omega$ is a J\'{o}nsson cardinal, then the following statements hold: 
  \begin{enumerate}
      \item Every infinite subset of $\Set{\omega_n}{n<\omega}$ is $\mathbf{\Sigma}_1$-stationary in $\omega_\omega$. 

      \item If $\map{r}{\omega_\omega}{\omega_\omega}$ is a regressive function that is definable by a $\Sigma_1$-formula with parameters in $\HH{\aleph_\omega}\cup\{\omega_\omega\}$, then $r$ is constant on an infinite subset of $\Set{\omega_n}{n<\omega}$. 
  \end{enumerate}
 \end{theorem}

 \begin{proof}
    Since $\omega_\omega$ is the least  J\'{o}nsson cardinal, we know that $\omega_\omega$ is $\omega_n$-Rowbottom for some $0<n<\omega$ (see {\cite[Proposition 8.15]{MR1994835}}). Then Lemma \ref{lemma:Rowbottem} implies that $\omega_\omega$ has the $\Sigma_1(\omega_k)$-undefinability property for all $n\leq k<\omega$. An application of Corollary \ref{corollary:UnbCardStat} then shows that every infinite subset of $\Set{\omega_m}{m<\omega}$ is $\mathbf{\Sigma}_1$-stationary in $\omega_\omega$.
    
    Now, assume that $\map{r}{\omega_\omega}{\omega_\omega}$ is a regressive function that is definable by a $\Sigma_1$-formula with parameters in $\HH{\aleph_\omega}\cup\{\omega_\omega\}$. Then we can find $0<n<\omega$ and $z\in\HH{\aleph_n}$ such that $\omega_\omega$ is $\omega_n$-Rowbottom and $r$ is definable by a $\Sigma_1$-formula $\varphi(v_0,\ldots,v_3)$ and the  parameters $\omega_\omega$ and $z$. By repeating the proof of Lemma \ref{lemma:Rowbottem}, we can find an elementary submodel $X$ of $\HH{\aleph_{\omega+1}}$ with $\tc{\{z\}}\cup\{\kappa\}\subseteq X$, $\betrag{X\cap\omega_\omega}=\aleph_\omega$ and $\betrag{X\cap\omega_n}=\aleph_{n-1}$. Let $\map{\pi}{X}{N}$ denote the corresponding transitive collapse. Then $\omega_\omega\in N$, $\pi(\omega_\omega)=\omega_\omega$ and $\omega_n\in\ran{\pi}$. Set $\lambda=\pi^{{-}1}(\omega_n)<\omega_n$. Then $\pi\restriction\lambda=\id_\lambda$ and an an easy induction shows that $\pi(\omega_k)<\omega_k$ holds for all $n\leq k<\omega$. In particular, we know that $\pi^{{-}1}(\alpha)>\alpha$ holds for all $\omega_n\leq\alpha<\omega_\omega$. 
    By elementarity, the formula $\varphi(v_0,\ldots,v_3)$ and the parameters $\omega_\omega$ and $z$ define a regressive function in $N$. In particular, $\Sigma_1$-upwards absoluteness implies that $\varphi(\alpha,r(\alpha),\omega_\omega,z)$ holds in $N$ for all $\alpha<\omega_\omega$. 
    We now define a strictly increasing sequence $\seq{\mu_i}{i<\omega}$ of ordinals in the interval $[\omega_n,\omega_\omega)$  by setting   $\mu_0=\omega_n$ and $\mu_{i+1}=\pi^{{-}1}(\mu_i)$ for all $i<\omega$. We can then inductively show that each $\mu_i$ is a cardinal   with $r(\mu_i)=r(\lambda)$ and hence $r$ is constant on an infinite subset of $\Set{\omega_m}{m<\omega}$. 
 \end{proof}

 We end this section by presenting an example of an application of the concepts isolated in this paper to reduce the class of models of set theory in which $\omega_\omega$ possesses strong partition properties. In particular, we will show that $\omega_\omega$ is not $\omega_2$-Rowbottom in the standard models of strong forcing axioms, where the given axiom was forced over a model of the $\GCH$ by turning some large cardinal into $\omega_2$. This implication will be a direct consequence of the following observation:

 \begin{lemma}
  Assume that there is a natural number $n_*>1$ such that there are no special $\omega_{n_*}$-Aronszajn trees and for all $n_*<n<\omega$,   there are special $\omega_n$-Aronszajn trees.  
  Then the set $\{\omega_{n_*}\}$ is definable by a $\Sigma_1$-formula with parameter $\omega_\omega$ and the cardinal $\omega_\omega$ is not $\omega_{n_*}$-Rowbottom. 
 \end{lemma}
 
\begin{proof}
Consider the collection of all transitive models $M$ of $\ZFC^-$ with the following properties:
\begin{itemize}
    \item $\omega_\omega+1\subseteq M$. 
    
    \item $\omega_\omega=\omega_\omega^M$. 
    
    \item In $M$, for every $n_*<n<\omega$, there is a special  $\omega^M_n$-Aronszajn tree. 
\end{itemize}
 The collection of such models $M$ is not empty as it includes $\HH{\aleph_{\omega+1}}$. 
 It is clear that for each model $M$ in this collection, we have $\omega_{n_*}^M = \omega_{n_*}$.
Hence, we can conclude that the set $\{\omega_{n_*}\}$ is definable by the $\Sigma_1$-formula with parameter $\omega_\omega$ that says there is a model $M$ with the above properties and $x$ is equal to $\omega_{n_*}^M$. 
 In particular, this shows that the cardinal $\omega_\omega$ does not have the $\Sigma_1(\omega_{n_*})$-undefinability property and hence Lemma \ref{lemma:Rowbottem} shows that $\omega_\omega$ is not $\omega_{n_*}$-Rowbottom.  
\end{proof}

 If we start in a model of the $\GCH$ containing a supercompact cardinal and use the canonical forcing to force the validity of some strong forcing axiom, like $\PFA$ or $\MM\MM^{++}$, then a result of Baumgartner shows that the tree property holds at $\omega_2$ and, since the $\GCH$ holds above $\aleph_0$, a result of Specker ensures that there are special $\omega_n$-Aronszajn trees for all $2<n<\omega$. Therefore, the above lemma shows that  $\omega_\omega$ is not $\omega_2$-Rowbottom in these models. 
 This observation should be compared with the consistency  results of K\"onig in \cite{zbMATH05132679}.


\section{Equiconsistency results}\label{Section:ConsistencyResults}

 We  build on the results of the previous  sections to obtain more equiconsistency results that witness various ways by which different types of $\Sigma_1$-definable closed unbounded sets of cardinals $\kappa$  fail to approximate the closed unbounded filter on $\kappa$.


\subsection{Limits of uncountably many measurable cardinals}

 We start by showing that, in the case of singular cardinals of uncountable cofinality, Theorem \ref{theorem:LimitMeasurables} provides the right consistency strength for both of the listed conclusions. In contrast, we will later show that the consistency strength for singular cardinals of countable cofinality is merely a single measurable cardinal.

 \begin{theorem}\label{theorem:UncCofManyMeasurables}
  Let $\kappa$ be a singular cardinal of uncountable cofinality. If there is no inner model with $\cof{\kappa}$-many measurable cardinals, then the following statements hold: 
  \begin{enumerate}
      \item Every subset of $\kappa$ that is $\mathbf{\Sigma}_1$-stationary in $\kappa$ is stationary in $\kappa$. In particular, there is an unbounded subset of $\kappa$ that consists of cardinals and is not $\mathbf{\Sigma}_1$-stationary in $\kappa$. 

      \item There exists a regressive function $\map{r}{\kappa}{\kappa}$ that is definable by a $\Sigma_1$-formula with parameters in $\HH{\kappa}\cup\{\kappa\}$ and is not constant on any unbounded subset of $\kappa$.
  \end{enumerate}
 \end{theorem}

 \begin{proof}
  By {\cite[Theorem 2.14]{MR926749}}, our assumptions imply that $0^{long}$  (as defined in {\cite[Definition 2.13]{MR926749}}) does not exist. 
  Let $U_{can}$ denote the canonical sequence of measures and let $\KK[U_{can}]$ denote the canonical core model (as defined in {\cite[Definition 3.15]{MR926749}}). Then our assumption ensures that $\dom{U_{can}}$ has order-type less than $\cof{\kappa}$. 
  Moreover, since $\cof{\kappa}$ is uncountable, we can apply    {\cite[Theorem 3.23]{MR926749}} to show that $\kappa$ is not measurable in $\KK[U_{can}]$. 
  But, this allows us to   use {\cite[Theorem 3.20]{MR926749}} to conclude that $\kappa$ is singular in $\KK[U_{can}]$. 
  Set $U=U_{can}\restriction\kappa$ and $\KK=\KK[U]$  (see {\cite[Definition 3.1]{MR926749}}). Then $U\in\HH{\kappa}$ and {\cite[Theorem 3.9]{MR926749}} shows that $\POT{\kappa}^{\KK[U_{can}]}\subseteq\KK$. In particular, we know that $\kappa$ is singular in $\KK$ and we can define $\map{c}{\cof{\kappa}^\KK}{\kappa}$ to be the $<_{\KK[U]}$-least cofinal function in $\KK$ (see {\cite[Theorem 3.4]{MR926749}}). Then {\cite[Lemma 2.3]{MR3845129}} shows that $c$ is definable by a $\Sigma_1$-formula with parameters in $\HH{\kappa}\cup\{\kappa\}$. But this shows that there is a closed unbounded subset $C$ of $\kappa$ of order-type $\cof{\kappa}^\KK$ such that $\min(C)>\cof{\kappa}^\KK$ and the set $\{C\}$ is definable by a $\Sigma_1$-formula with parameters in $\HH{\kappa}\cup\{\kappa\}$. 
  As in the proof of Corollary \ref{corollary:SingularUncCofStationary}, we now know that every subset of $\kappa$ that is $\mathbf{\Sigma}_1$-stationary in $\kappa$ is stationary in $\kappa$.  
  Finally, the function $$\Map{r}{\kappa}{\cof{\kappa}^\KK}{\alpha}{\otp{C\cap\alpha}}$$ is a regressive function that is definable by a $\Sigma_1$-formula with parameters in $\HH{\kappa}\cup\{\kappa\}$ and is not constant on any unbounded subset of $\kappa$. 
 \end{proof}

 By combining this result with Theorem \ref{theorem:LimitMeasurables}, we obtain the following equiconsistency:

 \begin{corollary}\label{corollary:StrengthSingularUncountCof}
     The following statements are equiconsistent over $\ZFC$: 
     \begin{enumerate}
         \item There exist uncountably many measurable cardinals. 

         \item There exists a singular cardinal $\kappa$ of uncountable cofinality with the property that some non-stationary subset of $\kappa$ is $\mathbf{\Sigma}_1$-stationary in $\kappa$. 

         \item There exists a singular cardinal $\kappa$ of uncountable cofinality with the property that some non-stationary subset of $\kappa$ is $\mathbf{\Sigma}_1(Ord)$-stationary in $\kappa$. \qed
     \end{enumerate}
 \end{corollary}


\subsection{Countable cofinalities}

The aim of this section is to prove the following equiconsistency result:

\begin{theorem}\label{Thm:Measurable}
 The following statements are equiconsistent over $\ZFC$:
 \begin{enumerate}
    \item\label{item:Thm:Measurable1} There is a measurable cardinal. 
    
    \item\label{item:Thm:Measurable2} Every unbounded subset of $\Set{\omega_n}{n<\omega}$ is $\mathbf{\Sigma}_1(Ord)$-stationary in $\omega_\omega$. 
    
    \item\label{item:Thm:Measurable3} There is a singular cardinal $\kappa$ of countable cofinality and  a subset of $\kappa$ that consists of cardinals and is $\mathbf{\Sigma}_1(Ord)$-stationary in $\kappa$. 

    \item\label{item:Thm:Measurable4} There is a singular cardinal $\kappa$ of countable cofinality and  a subset of $\kappa$ that consists of cardinals and is $\mathbf{\Sigma}_1$-stationary in $\kappa$. 
 \end{enumerate}
\end{theorem}

\begin{remark}
    The restriction to $\Sigma_1(Ord)$-definability in \eqref{item:Thm:Measurable2} of the theorem is optimal since the sequence $\seq{\omega_n}{n < \omega}$ is $\Sigma_2(Ord)$-definable. 
\end{remark}

In the following arguments we will make use of the Easton-support collapse version from \cite{EastonCollapse} of a universal collapse forcing (see \cite{MR2768692} for a comprehensive background).

\begin{definition}
Suppose that $\kappa$ is a Mahlo cardinal and let $$I_\kappa ~ = ~  \Set{ \gamma < \kappa}{\textit{$\gamma$ is inaccessible}}.$$ 
For each $\gamma\in I_\kappa$, we define $\qo(\gamma,\kappa)$ to be the product $$ \qo(\gamma,\kappa) = \prod_{\delta \in I_\kappa\setminus\gamma} \col(\delta,{<}\kappa)$$ with the Easton-support.
\end{definition}
%
\begin{remark}\label{Rmk:EastonCollapseBasics}
The partial order $\qo(\gamma,\kappa)$ is clearly definable from the ordinals $\gamma$ and $\kappa$, and forcing with $\qo(\gamma,\kappa)$ collapses all cardinals in the interval $(\gamma,\kappa)$. Since this partial order is ${<}\gamma$-closed and satisfies the $\kappa$-chain condition, the cardinal $\kappa$ becomes $\gamma^+$ in $\qo(\gamma,\kappa)$-generic extensions. 
In the following, we will rely on the following two useful features of this partial order: 
\begin{enumerate}
    \item  If $\gamma < \delta$ are elements of $I_\kappa$, then $\qo(\gamma,\delta) \times \qo(\delta,\kappa)$ is a regular sub-forcing of $\qo(\gamma,\kappa)$. 
    
    \item $\qo(\gamma,\kappa)$ is weakly homogeneous.
\end{enumerate} 
The proofs are left to the reader (see also \cite{EastonCollapse}). 
\end{remark}

Before proving  Theorem \ref{Thm:Measurable}, we review some basic facts about Prikry forcing and products of collapse forcings.  
Suppose that $\kappa$ is a measurable cardinal  and $U$ is a normal measure on $\kappa$. 
We let $\po_U$ denote the Prikry forcing given by $U$. 
Conditions $\po_U$ are then of the form $p = \langle s_p, A_p\rangle$,  where $s_p \in [\kappa]^{{<}\omega}$ is a finite sequence of inaccessible cardinals and $A_p \in U$ consists of inaccessible cardinals and satisfies  $\min(A_p) > \max(s_p)$.  
We then have $p\leq_{\po_U}q$   if $s_p$ end-extends $s_q$, $A_p \subseteq A_q$, and $s_p\setminus s_q \subseteq A_q$. 
If moreover, $s_p = s_q$ holds, then we say $p$ is a direct extension of $q$, denoted by $p \leq_{\po_U}^* q$.  
 Given a condition $p$ in $\po_U$ and $t\in[A_p]^{{<}\omega}$, we let $p\fr t$ denote the condition $\langle s_p\cup t,A_p\setminus(\max(t)+1)\rangle$. 
 Finally, if $G$ is $\po_U$-generic filter over $V$, then its associated Prikry sequence $\vec{\kappa}^G = \seq{\kappa^G_n}{n<\omega}$ is defined by $\vec{\kappa}^G = \bigcup \Set{s\in [\kappa]^{<\omega}}{\exists A \in U ~ \langle s,A\rangle \in G}$. 
Below, we record some of the basic properties of $\po_U$ that will be used in the following arguments (see \cite{MR2768695} for details).

\begin{fact}\label{Fact:PrikryForcing} ${}$
 \begin{enumerate}
  \item The direct extension order $\leq^*_{\po_U}$ is ${<}\kappa$-closed. 
  
  \item Forcing with $\po_U$ does not introduce new bounded subsets of $\kappa$. 
  
  \item The partial order $\po_U$ satisfies the $\kappa^+$-chain condition. 
  
 \item (\emph{Prikry Property}) For every condition $q$ in $\po_U$ and every statement $\sigma$ in the forcing language of $\po_U$, there is $p \leq^*_{\po_U} q$ which decides $\sigma$.  
 
 \item (\emph{Strong Prikry Property}) For every condition $q$ in $\po_U$ and every dense open set $D$ of $\po_U$, there is $p \leq_{\po_U}^* q$ and $n < \omega$ such that $D$ contains all conditions of the form $p\fr t$ with $t\in[A_p]^n$.  
 
 \item (\emph{Name Capturing Property}) For every condition $q$  in $\po_U$ and every $\po_U$-name $\name{f}$ for a function with domain $\omega$ with the property that $\name{f}^G(n)\in\HH{\kappa_n^G}$ holds whenever $n<\omega$ and $G$
 is $\po_U$-generic over $V$ with $q\in G$, 
 there is a condition  $p \leq_{\po_U}^* q$  and a function $\map{F}{[A_p]^{{<}\omega}}{\HH{\kappa}}$ so that 
$$ p\fr t \Vdash_{\po_U}\anf{\name{f}(\vert \check{s}_p \cup \check{t}\vert) = \can{F}(\check{t})}$$ holds for every $t \in [A^*]^{{<}\omega}$. 

\item A $\po_U$-generic filter $G$ is generated by its induced Prikry sequence  $\vec{\kappa}^G$ in the sense that $G = \Set{ \langle\vec{\kappa}^G\uhr n, A\rangle}{n<\omega, ~ \vec{\kappa}^G\setminus n \subseteq A\in U}$.

\item (\emph{Mathias Criterion}) A sequence $\vec{\kappa} = \seq{\kappa_n}{n<\omega}$ generates a $\po_U$-generic filter if and only if $\vec{\kappa}\setminus A$ is finite for every $A \in U$. 
\end{enumerate}
\end{fact}


To push the construction down to $\omega_\omega$, we force with a product of collapse posets after adding a Prikry forcing.  Let $\vec{\rho} = \seq{\rho_n}{n< \ell_\rho}$ be a strictly increasing sequence of inaccessible cardinals of length $0<\ell_\rho \leq \omega$ and define $\co_{\vec{\rho}}$ to be the  product $$\qo(\omega,{<}\rho_0) ~ \times ~ \prod_{1 \leq n < \ell_{\vec{\rho}}}\qo(\rho_{n-1},{<}\rho_{n})$$ with full support. Therefore, conditions  in $\co_{\vec{\rho}}$ are sequences
$q = \seq{q_n}{n < \ell_{\vec{\rho}}}$ with $q_0 \in \qo(\omega,{<}\rho_0)$ and $q_n \in \qo(\rho_{n-1},{<}\rho_{n})$ for all $0<n<\ell_{\vec{\rho}}$.

\begin{remark}\label{Remark:CO-WeaklyHomogeneous} 
${}$
\begin{enumerate}
 \item Standard arguments about product forcings show that, if  $G$ is ${\co_{\vec{\rho}}}$-generic over $V$, then $\rho_n=\omega_{n+1}^{V[G]}$ holds for every $n<\ell_{\vec{\rho}}$. 
 
 \item By the absorption argument for the Easton-support collapse product (Remark \ref{Rmk:EastonCollapseBasics}), 
 if  $s$ is a finite strictly increasing sequences of inaccessible cardinals and $t$ is a subsequence of $s$ with $\max(s) = \max(t)$, then $\co_s$ is a regular subforcing of $\co_t$.  This directly implies that, if $\vec{\rho}$ is a strictly increasing sequence of inaccessible cardinals of length $\omega$ and $s$ is a finite subset of $\vec{\rho}$, then $\co_{\vec{\rho}}$ is a regular subforcing of $\co_{\vec{\rho}\setminus s}$. Moreover, the associated forcing projection from  $\co_{\vec{\rho}\setminus s}$ to $\co_{\vec{\rho}}$ is the identity on the components below $\min(s)$ and above $\min\left(\vec{\rho} \setminus (\max(s)+1)\right)$.
 
 \item The forcing $\co_{\vec{\rho}}$ is weakly homogeneous. Therefore, if some condition in $\co_{\vec{\rho}}$ forces a statement with ground model parameters to hold, then every condition forces this statement to hold.  
Similarly, for every sequence $\vec{\rho}$ and an initial segment $s$, the quotient forcing $\co_{\vec{\rho}}/\co_{s}$ is also weakly homogeneous. 
\end{enumerate}
\end{remark}

In the following, let $\dot{\co}$ denote the canonical $\po_U$-name for a partial order with the property that $\dot{\co}^G=\co_{\vec{\kappa}^G}$ holds whenever $G$ is $\po_U$-generic over $V$. 
The argument of the following Lemma will give the forcing direction of Theorem \ref{Thm:Measurable}.

\begin{lemma}\label{Lemma:Main-Prikry}
  If $G * H$  is $(\po_U*\dot{\co})$-generic  over $V$  and $\gamma\geq\kappa$ is an ordinal, then, in $V[G,H]$, the cardinal $\omega_\omega$ has the $\Sigma_1(\omega_n,\gamma)$-undefinability property for all sufficiently large natural numbers $n$. 
\end{lemma}

\begin{proof}
  Fix a condition $p_*$ in $\po_U$.   For each natural number $n\geq\vert s_{p_*}\vert$, let $\sigma_n$ be the statement in the forcing language of $\po_U$ that says that there is a condition in $\dot{\co}$ which forces that $\kappa$ does not have the $\Sigma_1(\dot{\kappa}_n,\gamma)$-undefinability property, where $\dot{\kappa}_n$ denotes the canonical $\po_U$-name for $\kappa_n^G$. 
  By the Prikry property, for each $\vert s_{p_*}\vert\leq n<\omega$,  there is a condition $p_n\leq^*_{\po_U}p_*$ that  decides  $\sigma_n$. We complete the proof by showing that, for all $\vert s_{p_*}\vert\leq n<\omega$, the condition $p_n$ forces $\neg \sigma_n$ to hold. 
 Suppose otherwise that $p_n\Vdash_{\po_U}\sigma_n$ for some $\vert s_{p_*}\vert\leq n<\omega$. Pick a condition $p\leq_{\po_U}p_n$ with $\betrag{s_p}=n$. 
 Since $p \Vdash_{\po_U} \sigma_n$, there is a $\po_U$-name $\name{q}$ for a condition in $\dot{\co}$ so that $$\langle p,\name{q}\rangle \Vdash_{\po_U * \dot{\co}} \anf{\textit{$\kappa$ does not  have the $\Sigma_1(\name{\kappa}_n,\check{\gamma})$-undefinability property}}.$$ Since $\name{\kappa}_n$ is a $\po_U$-name and $\dot{\co}$ is forced to be weakly homogeneous (see Remark \ref{Remark:CO-WeaklyHomogeneous}),  we can assume that $\name{q}$ is the name for the trivial condition $\dot{\mathbbm{1}}_{\dot{\co}}$ in  $\dot{\co}$. 
This allows us to find  $(\po_U * \dot{\co})$-names $\name{x}$, $\name{\alpha}$ and $\name{\tau}$ so that $\langle p,\dot{\mathbbm{1}}_{\dot{\co}}\rangle$ forces the following statements to hold: 
\begin{itemize}
 \item $\name{x}$ is an element of $\HH{\dot{\kappa}_n}$.  

 \item $\name{\alpha}$ is an ordinal in the interval $[\dot{\kappa}_n,\check{\kappa})$. 
 
 \item $\dot{\tau}$ is the G\"odel number of a $\Sigma_1$-formula $\varphi(v_0,\ldots,v_3)$ that defines the set $\{\name{\alpha}\}$ using the parameters $\check{\kappa}$, $\check{\gamma}$ and $\name{x}$. 
\end{itemize}
 
 Making another use of the fact that $\name{\co}$ is forced to be a weakly homogeneous partial order, we may assume $\name{\alpha}$ and $\name{\tau}$ are $\po_U$-names. 
 Moreover, using the Prikry property of $\po_U$, we may also assume that $p$ decides that $\name{\tau}$ codes a given $\Sigma_1$-formula $\varphi(v_0,\ldots,v_3)$.  
 Since the quotient forcing $\co_{\vec{\kappa}^G}/\co_{\vec{\kappa}^G\restriction(n+1)}$ is weakly homogeneous whenever $G$ is $\po_U$-generic over $V$, we may assume that $\name{x}$ is a $(\po_U*\name{\co}_n)$-name, where $\name{\co}_n$ is the canonical $\po_U$-name for $\co_{\vec{\kappa}^G\restriction(n+1)}$. 
 In addition, since forcing with $\po_U$ does not add new bounded sets to $\kappa$ and $\name{\co}_n$ is forced to satisfy the $\name{\kappa}_n$-chain condition, there is a $\po_U$-name $\dot{y}$ with the property that whenever $G$ is $\po_U$-generic over $V$ with $p\in G$, then $\name{y}^G$ is a $\co_{\vec{\kappa}^G\restriction(n+1)}$-name in $\HH{\kappa_n^G}^V$ that $\co_{\vec{\kappa}^G\restriction(n+1)}$ forces to be equal to $\dot{x}$.
 The fact that   $p\fr\langle\rho\rangle\Vdash_{\po_U}\anf{\name{\kappa}_n=\check{\rho}}$ holds  for every $\rho \in A_p$ allows us to use the  name capturing property of $\po_U$ (see Fact \ref{Fact:PrikryForcing}) to find a function $\map{Y}{A_p}{\HH{\kappa}}$ with the property that for every $\rho \in A_p$, the set $Y(\rho)$ is a $\co_{s_p\cup\{\rho\}}$-name in $\HH{\rho}$ with $p\fr\langle\rho\rangle\Vdash_{\po_U}\anf{\name{y}=\check{Y}(\check{\rho})}$.  
 Using the normality of $U$, we can find $A\subseteq A_p$ in $U$ and $\dot{y}_0$ in $\HH{\kappa}$ with the property that $Y(\rho)=\dot{y}_0$ holds for all $\rho\in A$. 
 Let $\rho_0 = \min(A)$. Then we can find an ordinal $\alpha$ in the interval $[\rho_0,\kappa)$ and a condition $r_0\leq_{\po_U}p\fr\langle\rho_0\rangle$ with $\alpha<\max(s_{r_0})$ and $r_0\Vdash_{\po_U}\anf{\dot{\alpha}=\check{\alpha}}$. 
 Set $\rho_1=\min(A\cap A_{r_0})$ and  
$r_1=\langle s_p\cup\{\rho_1\},A_{r_0}\setminus(\rho_1+1)\rangle$. Then $r_1$ is a condition in $\po_U$ that strengthens   $p\fr\langle\rho_1\rangle$. Let $G_1*H_1$ be $(\po_U*\dot{\co})$-generic over $V$ with $r_1\in G_1$. We then have $\kappa_n^{G_1}=\rho_1>\max(s_{r_0})>\alpha$. 
By the Mathias criterion for $\po_U$ (see Fact \ref{Fact:PrikryForcing}), the sequence $s_{r_0}\fr(\vec{\kappa}^{G_1}\restriction[n,\omega))$ generates a $\po_U$-generic filter $G_0$ over $V$ that is an element of $V[G_1]$. 
It is then clear that $r_0$ is an element of $G_0$ and $V[G_0]=V[G_1]$. Moreover, since $\vec{\kappa}^{G_1}\restriction n=s_p=s_{r_0}\restriction n$, it follows that $\vec{\kappa}^{G_1}$ is a subsequence of $\vec{\kappa}^{G_0}$ with finite difference and hence the absorption property of Easton collapse posets   (see Remark \ref{Rmk:EastonCollapseBasics}) ensures that  the partial order  $\co_{\vec{\kappa}^{G_0}}$ is a regular subforcing of the partial order  $\co_{\vec{\kappa}^{G_1}}$. Let $H_0\in V[G_0,H_1]$ denote the filter on $\co_{\vec{\kappa}^{G_0}}$ induced by $H_1$. 
 Since $H_0$ and $H_1$ induce the same filter on $\co_{s_p\cup\{\rho_0\}}$, we know that  $\dot{y}_0^{H_0}=\dot{y}_0^{H_1}$. 
 This implies that $\varphi(\alpha,\kappa,\gamma,\dot{y}_0^{H_1})$ holds in $V[G_0,H_0]$ and, by $\Sigma_1$-upwards absoluteness, this statement also holds in $V[G_0,H_1]$. 
 But this yields a contradiction, because the fact that $r_1$ is an element of $G_1$ implies that $\dot{\alpha}^{G_1}\geq\kappa_n^{G_1}=\rho_1>\alpha$ and $\dot{\alpha}^{G_1}$ is the unique element $a$ of $V[G_0,H_1]$ with the property that $\varphi(a,\kappa,\gamma,\dot{y}_0^{H_1})$ holds. 
\end{proof}

 A combination of the above lemma with Corollary \ref{corollary:UnbCardStat} now directly yields the following result:

\begin{corollary}\label{corollary:CardStatAtOmegaOmega}
 If $G * H$  is $(\po_U*\dot{\co})$-generic  over $V$, then, in $V[G,H]$, every unbounded subset of $\Set{\omega_n}{n<\omega}$ is $\mathbf{\Sigma}_1(\On)$-stationary in $\omega_\omega$. \qed 
\end{corollary}

 We are now ready to prove the main result of this section:

\begin{proof}[Proof of Theorem \ref{Thm:Measurable}]
 First, let $\kappa$ be a singular cardinal of countable cofinality and let $E$ be a subset of $\kappa$ that consists of cardinals and is $\mathbf{\Sigma}_1(\On)$-stationary in $\kappa$. Assume, towards contradiction, that there is no inner model with a measurable cardinal. Then $\kappa$ is singular in the  Dodd-Jensen core model $\KK^{DJ}$. 
 By Lemma \ref{lemma:LeastCofinal}, there is a 
cofinal function $\map{c}{\cof{\kappa}^{\KK^{DJ}}}{\kappa}$ which is definable by a $\Sigma_1$-formula with parameter $\kappa$. In addition, pick a cofinal function $\map{g}{\omega}{\cof{\kappa}^{\KK^{DJ}}}$. 
Then the composition $\map{f = c \circ g}{\omega}{\kappa}$ is cofinal in $\kappa$ and $\Sigma_1$-definable from the parameters $\kappa$ and $g \in \HH{\kappa}$. 
Then the $C = \Set{ f(n)+1}{n < \omega}$ is closed unbounded in $\kappa$ and it is definable by a $\Sigma_1$-formula with parameters in $\HH{\kappa}\cup\{\kappa\}$. Since it consists of successor ordinals, it is disjoint from $E$, contradicting our assumptions. 

 The above computations show that \eqref{item:Thm:Measurable4} implies \eqref{item:Thm:Measurable1} in the statement of the theorem. This completes the proof of the theorem, because  the implications from \eqref{item:Thm:Measurable2} to \eqref{item:Thm:Measurable3} and from \eqref{item:Thm:Measurable3} to \eqref{item:Thm:Measurable4} are trivial and the implication from \eqref{item:Thm:Measurable1} to \eqref{item:Thm:Measurable2} is given by Corollary \ref{corollary:CardStatAtOmegaOmega}. 
 %
\end{proof}

By combining Theorem \ref{Thm:Measurable} with the second part of Corollary \ref{corollary:NoMeasurableCorollaries3}, we derive the following equiconsistency result that shows that the statement of Proposition \ref{proposition:CountCofBistat} is optimal in $\ZFC$.

\begin{corollary}\label{Cor:EquiconsistentReaping}
 The following statements are equiconsistent over $\ZFC$:     
 \begin{enumerate}
     \item There is a measurable cardinal. 

     \item There is a singular cardinal $\kappa$ of countable cofinality with the property that for every subset $A$ of $\HH{\kappa}$ of cardinality $\mathfrak{r}$, there exists a subset $E$ of $\kappa$ such that both $E$ and $\kappa\setminus E$ are $\Sigma_1(A)$-stationary. 

     \item There is a singular cardinal $\kappa$ of countable cofinality with the property that there exists a subset $E$ of $\kappa$ such that both $E$ and $\kappa\setminus E$ are $\mathbf{\Sigma}_1(Ord)$-stationary. \qed 
 \end{enumerate}
\end{corollary}


\subsection{Successors of regular  cardinals}

We continue by studying $\Sigma_1(A)$-stationary subsets of successor cardinals. In this section, we prove the following equiconsistency result for successors of regular cardinals:

\begin{theorem}\label{Thm:Mahlo}
The following statements are equiconsistent over $\ZFC$:
\begin{enumerate}
    \item\label{item:Thm:Mahlo1} There is a Mahlo cardinal. 
    
    \item\label{item:Thm:Mahlo2} There is a regular cardinal $\mu$ with the property that the set $\{\mu\}$ is $\Sigma_1(\HH{\mu})$-stationary in $\mu^+$. 

    \item\label{item:Thm:Mahlo3} There is a regular cardinal $\mu$ with the property that the set $\{\mu\}$ is $\Sigma_1$-stationary in $\mu^+$. 
\end{enumerate}
\end{theorem}

\begin{lemma}\label{lemma:MahloSilverCollapse}
  If $\kappa$ is a Mahlo cardinal, then there exists an inaccessible cardinal $\delta<\kappa$ with the property that   $\qo(\delta,\kappa)$ forces $\kappa$ to have the $\Sigma_1(\delta)$-undefinability property. 
\end{lemma}

\begin{proof}
    Assume, towards a contradiction, that no such $\delta$ exists. Then for every $\gamma\in I_\kappa$, the fact that $\qo(\gamma,\kappa)$ is weakly homogeneous (see Remarks \ref{Rmk:EastonCollapseBasics} and \ref{Remark:CO-WeaklyHomogeneous}) allows us to find $x_\gamma$, $\alpha_\gamma$ and $\tau_\gamma$ such that the following statements hold: 
    \begin{itemize}
        \item $x_\gamma$ is an element of $\HH{\gamma}$. 

        \item $\alpha_\gamma$ is an ordinal in the interval $[\gamma,\kappa)$. 

        \item $\tau_\gamma$ is a G\"odel number of a $\Sigma_1$-formula $\varphi(v_0,v_1,v_2)$ with the property that the trivial condition of $\qo(\gamma,\kappa)$ forces that $\alpha_\gamma$  is the unique set $a$ such that $\varphi(a,\kappa,x_\gamma)$ holds. 
    \end{itemize}
    
 Since $\kappa$ is a Mahlo cardinal, we can find a stationary subset $S$ of $I_\kappa$ with the property that there is an element $x$ of $\HH{\kappa}$ and a $\Sigma_1$-formula $\varphi(v_0,v_1,v_2)$ such that for all $\gamma\in S$, we have $x_\gamma=x$ and $\tau_\gamma$ is a G\"odel number of $\varphi(v_0,v_1,v_2)$.  
 Let $\gamma_0 = \min(S)$ and $\gamma_1 = \min(S \setminus (\alpha_{\gamma_0}+1))$. Clearly, we then have $\gamma_0<\gamma_1$ and $\alpha_{\gamma_1} > \alpha_{\gamma_0}$. 
 Let $G_0$ be $\qo(\gamma_0,\kappa)$-generic over $V$. 
 As $\qo(\gamma_1,\kappa)$ is a regular subforcing of $\qo(\gamma_0,\kappa)$, we can now find $G_1\in V[G_0]$ that is $\qo(\gamma_1,\kappa)$-generic over $V$. In this situation, we know that $\varphi(\alpha_{\gamma_1},\kappa,x)$ holds in $V[G_1]$ and, by $\Sigma_1$-upwards absoluteness, this statement also holds in $V[G_0]$. But this yields a contradiction, because $\alpha_{\gamma_1}>\alpha_{\gamma_0}$ and $\alpha_{\gamma_0}$ is the unique element $a$ of $V[G_0]$ with the property that $\varphi(a,\kappa,x)$ holds in $V[G_0]$.  
\end{proof}

\begin{proof}[Proof of Theorem \ref{Thm:Mahlo}]
  Suppose that $\mu$ is a regular cardinal so that the set $\{\mu\}$ is $\Sigma_1$ stationary in $\mu^+$. 
 Assume, towards a contradiction, that $\kappa=\mu^+$ is not a Mahlo cardinal in $\LL$. Then there is a constructible closed unbounded subset of $\kappa$ whose elements are singular in $\LL$. Let $C$ denote the $<_\LL$-least such subset of $\kappa$. Then the set $\{C\}$ is definable by a $\Sigma_1$-formula with parameter $\kappa$ and hence we know that $\mu$ is an element of $C$. But this is a contradiction, because all elements of $C$ are singular. 
 
 The above computations yield the implication from \eqref{item:Thm:Mahlo3} to \eqref{item:Thm:Mahlo1} in the statement of the theorem. The implication from \eqref{item:Thm:Mahlo1} to \eqref{item:Thm:Mahlo2} is given by a combination of Lemma  \ref{lemma:UnDefStationary} and Lemma \ref{lemma:MahloSilverCollapse}. Finally, the implication from \eqref{item:Thm:Mahlo2} to \eqref{item:Thm:Mahlo3} is trivial.  
\end{proof}


\subsection{Successors of singular cardinals}

The last arguments used in the above proof of Theorem \ref{Thm:Mahlo} do not apply if we consider successors of singular cardinals. Indeed, it turns out that the corresponding assumptions has much higher consistency strength. In one direction, we show that an analogous statement holds for successors of limits of measurable cardinals:

\begin{theorem}\label{Thm:LimitOfMeasurables}
If $\kappa$ is a limit of measurable cardinals, then the cardinal $\kappa^+$ has the $\Sigma_1(\kappa)$-undefinability property.  
\end{theorem}

\begin{proof}
  Assume, towards a contradiction, that there is a $\Sigma_1$-formula $\varphi(v_0,v_1,v_2)$, an ordinal $\alpha$ in the interval $[\kappa,\kappa^+)$ and an element $z$ of $\HH{\kappa}$ such that $\alpha$ is the unique set $a$ with the property that $\varphi(a,\kappa^+,z)$ holds. Pick a measurable cardinal $\delta<\kappa$ such that $z\in\HH{\delta}$ and $\cof{\kappa}\in\delta\cup\{\kappa\}$. Pick a normal ultrafilter $U$ on $\delta$ and let $$\langle\seq{M_\alpha}{\alpha\in\On},\seq{\map{j_{\alpha,\beta}}{M_\alpha}{M_\beta}}{\alpha\leq\beta\in\On}\rangle$$ denote the linear iteration of $\langle V,\in,U\rangle$. Standard arguments now allow us to conclude that  $j_{0,\gamma}(\kappa^+)=\kappa^+$ and $j_{0,\gamma}(z)=z$ holds for all $\gamma<\kappa^+$ (see, for example, {\cite[Lemma 7.3]{Sigma_1_higher}}). We can now pick $\gamma<\kappa^+$ with $j_{0,\gamma}(\delta)>\alpha$. Elementarity then implies that $\varphi(j_{0,\gamma}(\alpha),\kappa^+,z)$ holds in $M_\gamma$ and, by $\Sigma_1$-upwards absoluteness, this shows that $\varphi(j_{0,\gamma}(\alpha),\kappa^+,z)$ also holds in $V$. But, this yields a contradiction, because $j_{0,\gamma}(\alpha)>\alpha$.  
 \end{proof}

 In combination with Lemma \ref{lemma:UnDefStationary}, the above result shows that if $\kappa$ is a limit of measurable cardinals, then the set $\{\kappa\}$ is $\Sigma_1(\HH{\kappa})$-stationary in $\kappa^+$. We end this section by showing that, in the case of successors of singular cardinals, the used large cardinal assumption is optimal. In the proof of this result, we again rely on the results of \cite{MR926749}.

 \begin{theorem}
     Let $\kappa$ be a singular cardinal with the property that the set $\{\kappa\}$ is $\Sigma_1(\HH{\kappa})$-stationary in $\kappa^+$. Then there is an inner model with $\cof{\kappa}$-many measurable cardinals. 
 \end{theorem}

 \begin{proof}
     Assume, towards a contradiction, that the above conclusion fails. An application of {\cite[Theorem 2.14]{MR926749}} then shows that $0^{long}$ does not exist. 
  Let $U_{can}$ denote the canonical sequence of measures and let $\KK[U_{can}]$ denote the canonical core model. Our assumption then implies that $\dom{U_{can}\restriction\kappa}$ is a bounded subset of $\kappa$. 

  \begin{claim*}
      $\kappa^+=(\kappa^+)^{\KK[U_{can}]}$.
  \end{claim*}

  \begin{proof}[Proof of the Claim]
   First, if $\kappa\notin\dom{U_{can}}$, then the fact that $\dom{U_{can}\restriction\kappa}$ is  bounded in $\kappa$ allows us to use {\cite[Theorem 3.20]{MR926749}} to derive the desired conclusion. Hence, we may assume that $\kappa$ is an element of $\dom{U_{can}}$. An application of {\cite[Theorem 3.23]{MR926749}} then shows that $\cof{\kappa}=\omega$ and there is a generic extension of $\KK[U_{can}]$ by finitely-many Prikry forcings that computes $\kappa^+$ correctly. But this also means that  $\KK[U_{can}]$ computes $\kappa^+$ correctly. 
  \end{proof}

  Set $U=U_{can}\restriction\kappa\in\HH{\kappa}$. By {\cite[Theorem 3.9]{MR926749}}, we then have $\POT{\kappa}^{\KK[U_{can}]}\subseteq\KK[U]$ and hence the above claim shows that $\kappa^+=(\kappa^+)^{\KK[U]}$. 

  \begin{claim*}
     The set $\{\kappa\}$ is definable by a $\Sigma_1$-formula with parameters $\kappa^+$ and $U$.
  \end{claim*}

  \begin{proof}[Proof of the Claim]
    If $\zeta<\kappa^+$ is an ordinal and  $M$ is a $U$-mouse (i.e., an iterable premouse over $U$, see {\cite[Definition 2.5]{MR926749}}) with $\kappa^+\in lp(M)$ and $\kappa^+=(\zeta^+)^M$, then a direct adaptation of the proof of Lemma \ref{lemma:InitialSegmentsKDJ}  shows that $\POT{\kappa}\subseteq M$ and hence $\kappa=\zeta$. Since there exists a $U$-mouse that satisfies the listed properties with respect to $\kappa$, this observation yields the desired definition of the set $\{\kappa\}$. 
  \end{proof}

  The above claim directly yields a contradiction, because it implies that the set $\{(\kappa,\kappa^+)\}$ is definable by a $\Sigma_1$-formula with parameters in $\HH{\kappa}\cup\{\kappa^+\}$. 
 \end{proof}


\section{Open Problems}\label{Section:Concluding Remarks and Open Problems}

There are many natural ways to vary Definition \ref{definition:UnDefProp}. For example, given a cardinal $\kappa\geq\omega_2$, we may ask whether for some uncountable ordinal $\alpha<\kappa$, the set $\{\alpha\}$ is definable by a $\Sigma_1$-formula with parameter $\kappa$. Since our arguments to derive consistency strength from the undefinability property (as in the proof of Theorem \ref{Thm:Measurable}) cannot be directly adjusted to this variation, we arrive at the following question:

\begin{question}
  Assume that for every uncountable ordinal $\alpha<\omega_\omega$, the set $\{\alpha\}$ is not definable by a $\Sigma_1$-formula with parameter $\omega_\omega$. Does $0^\#$ exist?
\end{question}

In Section \ref{subsection:PartitionProperty}, we show that many of the implications of large cardinals on $\Sigma_1(A)$-stationary sets can also be derived for smaller cardinals in the case where these cardinals possess strong partition properties.  
 For some of these implications, it is natural to ask whether they can be strengthened. First, since Lemma \ref{lemma:Rowbottem} relies on additional assumptions on the given Rowbottom cardinal, we ask whether these assumptions can be omitted:

 \begin{question}
   Does every $\mu$-Rowbottom cardinal have the $\Sigma_1(\mu)$-undefinability property?   
 \end{question}

 Second, when we consider $\Sigma_1$-definable regressive functions and compare the third part of Theorem \ref{theorem:StablyMeasurable} and the second part of Theorem \ref{theorem:LimitMeasurables} with the second part of Theorem \ref{theorem:JonssonStationary}, then we notice that our result for the case where $\omega_\omega$ is J\'{o}nsson is restricted to $\mathbf{\Sigma}_1$-stationary sets consisting of cardinals. We therefore ask if the given conclusion can also be extended to arbitrary   $\mathbf{\Sigma}_1$-stationary subsets in this setting.

 \begin{question}
   Assume that $\omega_\omega$  is a  J\'{o}nsson cardinal, $S$ is a $\mathbf{\Sigma}_1$-stationary subset of $\omega_\omega$ and $\map{r}{\omega_\omega}{\omega_\omega}$ is a regressive function that is definable by a $\Sigma_1$-formula with parameters in $\HH{\aleph_\omega}\cup\{\omega_\omega\}$. Is $r$ constant on a $\mathbf{\Sigma}_1$-stationary subset of $S$? 
 \end{question}

Finally, let us say that a cardinal $\kappa$ is \emph{strongly measurable with respect to $\Sigma_1(Ord)$-clubs} if there is an inner model $W$ in which $\kappa$ is measurable  such that the collection of $\Sigma_1(Ord)$-closed unbounded subsets of $\kappa$ (in $V$) is generated by the intersection filter $F \in W$ of $\eta < \kappa$ many $\kappa$-complete ultrafilters on $\kappa$ in $W$. The proof of Theorem \ref{Thm:Measurable} shows that $\omega_\omega$ can be strongly measurable with respect to $\Sigma_1$-definable clubs. Can the same hold for other singular cardinals such as $\omega_{\omega_1}$?

 \begin{question}
   Can $\omega_{\omega_1}$ be strongly measurable with respect to $\Sigma_1(Ord)$-clubs. 
 \end{question}

See \cite{BNU} for an analog of Prikry forcing, that changes the cofinality of a cardinal $\kappa$ to $\omega_1$ by a  homogeneous poset.

${}$\\ \textbf{Acknowledgments:}
The authors would like to thank Matt Foreman, Yair Hayut and Menachem Magidor for their comments, and to an anonymous referee for valuable suggestions and comments.





 \bibliographystyle{plain}
 \bibliography{references}

\end{document}